\documentclass{amsart}

\usepackage[T1]{fontenc}
\usepackage{hyperref}
\usepackage{amscd}
\usepackage[all]{xy}

\usepackage{multicol}
\title{Differential geometry of holomorphic vector bundles on a curve}
\author{Florent Schaffhauser}
\address{Departamento de Matem\'aticas, Universidad de Los Andes, Bogot\'a, Colombia.}
\email{florent@uniandes.edu.co}

\newcommand{\GL}{\mathbf{GL}}
\newcommand{\U}{\mathbf{U}}
\newcommand{\bK}{\mathbf{K}}

\newcommand{\K}{\mathbb{K}}
\newcommand{\R}{\mathbb{R}}
\newcommand{\C}{\mathbb{C}}
\newcommand{\Z}{\mathbb{Z}}
\newcommand{\Q}{\mathbb{Q}}

\newcommand{\cU}{\mathcal{U}}
\newcommand{\cM}{\mathcal{M}}
\newcommand{\cN}{\mathcal{N}}
\newcommand{\cE}{\mathcal{E}}
\newcommand{\cG}{\mathcal{G}}
\newcommand{\cA}{\mathcal{A}}
\newcommand{\cL}{\mathcal{L}}
\newcommand{\cF}{\mathcal{F}}
\newcommand{\cC}{\mathcal{C}}

\newcommand{\fk}{\mathfrak{k}}

\newcommand{\Hom}{\mathrm{Hom}}
\newcommand{\Aut}{\mathrm{Aut}}
\newcommand{\Gr}{\mathrm{Gr}}
\newcommand{\Id}{\mathrm{Id}}
\newcommand{\End}{\mathrm{End}\,}
\newcommand{\Dol}{\mathrm{Dol}}
\newcommand{\gr}{\mathrm{gr}}

\renewcommand{\phi}{\varphi}
\renewcommand{\frak}{\mathfrak}
\renewcommand{\O}{\mathcal{O}_{\Si_g}}
\renewcommand{\deg}[1]{\mathit{deg}\,#1}
\renewcommand{\Im}{\mathrm{im}\,}

\newcommand{\lra}{\longrightarrow}
\newcommand{\lmt}{\longmapsto}
\newcommand{\RP}{\R\mathbf{P}}
\newcommand{\CP}{\C\mathbf{P}}
\newcommand{\ceH}{\check{H}}
\newcommand{\Ga}{\Gamma}
\newcommand{\Si}{\Sigma}
\newcommand{\si}{\sigma}
\newcommand{\Om}{\Omega}
\newcommand{\w}{\omega}
\newcommand{\del}{\partial}
\newcommand{\delb}{\overline{\partial}}
\newcommand{\Jac}{\mathcal{J}ac\,}
\newcommand{\Pic}{\mathcal{P}ic\,}
\newcommand{\Mod}{\mathcal{M}_{\Si_g}}
\newcommand{\Mods}{\cN_{\Si_g}}
\newcommand{\wred}{\w^{\mathrm{red}}}
\newcommand{\umu}{\vec{\mu}}
\newcommand{\vol}{\mathrm{vol}}
\newcommand{\tr}{\mathrm{tr}}
\newcommand{\fu}{\mathfrak{u}}
\newcommand{\gl}{\mathfrak{gl}}

\newcommand{\rk}[1]{\mathrm{rk}\, #1}
\newcommand{\Ad}[1]{\mathrm{Ad}_{#1}\,}
\newcommand{\ov}[1]{\overline{#1}}
\newcommand{\Adet}[1]{\mathrm{Ad}^*_{#1}\,}
\newcommand{\fibre}{\mu^{-1}(\{0\})}

\newtheorem{theorem}{Theorem}[section]
\newtheorem{proposition}[theorem]{Proposition}
\newtheorem{lemma}[theorem]{Lemma}
\newtheorem{corollary}[theorem]{Corollary}

\theoremstyle{definition}
\newtheorem{definition}[theorem]{Definition}

\newtheorem{example}[theorem]{Example}
\newtheorem{examples}{Examples}

\newtheoremstyle{exercise}{\topsep}{5pt}%
     {}
     {}
     {\scshape}
     {.}
     {\topsep}
     {\thmname{#1}\thmnumber{\,#2}\thmnote{\,(#3)}}

\numberwithin{equation}{section}

\theoremstyle{exercise}
\newtheorem{exercise}{Exercise}[section]

\begin{document}

\begin{abstract}
These notes are based on a series of five lectures given at the 2009 \textit{Villa de Leyva Summer School on Geometric and Topological Methods for Quantum Field Theory}. The purpose of the lectures was to give an introduction to differential-geometric methods in the study of holomorphic vector bundles on a compact connected Riemann surface, as examplified in the celebrated paper of Atiyah and Bott (\cite{AB}). In these notes, we take a rather informal point of view and try to paint a global picture of the various notions that come into play in that study, setting Donaldson's theorem on stable holomorphic vector bundles (\cite{Don_NS}) as a goal for the lectures.
\end{abstract}

\maketitle

\tableofcontents

\section{Holomorphic vector bundles on Riemann surfaces}\label{Hol_VB}

\subsection{Vector bundles}

\subsubsection{Definition}

We begin by recalling the definition of a vector bundle. Standard references for the general theory of fibre bundles are the books of Steenrod (\cite{Steenrod}) and Husemoller (\cite{Husemoller}).

\begin{definition}[Vector bundle]\index{Vector bundle}\label{def:vector_bundle}
Let $X$ be a topological space and let $\K$ be the field $\R$ or $\C$. A \textbf{topological} $\K$\textbf{-vector bundle on} $X$ is a continuous map $p:E\lra X$ such that 

\begin{enumerate}

\item $\forall x \in X$, the fibre $E_x:= p^{-1}(\{x\})$ is a finite-dimensional $\K$-vector space,

\item $\forall x \in X$, there exists an open neighbourhood $U$ of $x$ in $X$, an integer $r_U \geq 0$, and a homeomorphism $\phi_U$ such that the diagramme

$$
\xymatrix{
p^{-1}(U) \ar[rr]^{\simeq}_{\phi_U} \ar[dr]_p & & U \times \K^{r_U} \ar[dl]^{\mathrm{pr}_U} \\
& U 
}
$$

\noindent is commutative ($\mathrm{pr}_U$ denotes the projection onto $U$),

\item the induced homeomorphism $\phi_x: p^{-1}(\{x\}) \lra \{x\} \times \K^{r_U}$ is a $\K$-linear isomorphism.

\end{enumerate}

\end{definition}

Most of the times, the map $p$ is understood, and we simply denote $E$ a vector bundle on $X$. If $\K=\R$, $E$ is called a (topological) \textit{real} vector bundle, and if $\K=\C$, it is called a (topological) \textit{complex} vector bundle. $p^{-1}(\{U\})$ is also denoted $E|_U$ and $\phi_U$ is called a \textit{local trivialisation} of $E$ over $U$. $E|_U$ is itself a vector bundle (on $U$). The open set $U\subset X$ is said to be trivialising for $E$, and the pair $(U,\phi_U)$ is called a (bundle) \textit{chart}.

\begin{examples}\label{examples_of_vb} The following maps are examples of vector bundles.

\begin{enumerate}

\item The \textit{product bundle} $p:X \times \K^r \lra X$, where $p$ is the projection onto $X$.

\item The tangent bundle $TM \lra M$ to a differentiable manifold $M$.

\item The M\"obius bundle on $S^1$: let $\cM$ be the quotient of $[0;1] \times \R$ under the identifications $(0,t) \sim (1,-t)$, with projection map $p: \cM \lra S^1$ induced by the canonical projection $[0;1]\times \R \lra [0;1]$. Observe that $\cM$ is indeed homeomorphic to a M\"obius band without its boundary circle.

\item The canonical line bundle on the $n$-dimensional projective space $\RP^n$ (=the space of lines in $\R^{n+1}$): $$E_{\mathrm{can}} :=\{(\ell,v)\in \RP^n\times \R^{n+1}\ |\ v\in \ell\}$$ with projection map $p(\ell,v)=l$. The fibre of $p$ above $\ell$ is canonically identified with $\ell$. When $n=1$, the bundle $E_{\mathrm{can}}$ will be shown later to be isomorphic to the M\"obius bundle (Exercise \ref{Mobius}). The same example works with $\CP^n$ in place of $\RP^n$.

\item The Grassmannian of $k$-dimensional complex sub-spaces of $\C^{n+1}$, denoted $\Gr_k(\C^{n+1})$, has a complex vector bundle structure with $k$-dimensional fibres: $$E_{\mathrm{can}} = \{(F,v) \in \Gr_k(\C^{n+1}) \times \C^{n+1}\ |\ v \in F\}$$ with projection map $p(F,v) = F$. The fibre of $p$ above $F$ is canonically identified with $F$. The same example works with $\R^{n+1}$ in place of $\C^{n+1}$.

\end{enumerate}

\end{examples}

It follows from the definition of a vector bundle that the $\Z_+$-valued map $$x \lmt \rk{p^{-1}(\{x\})}$$ (called the \textit{rank function}) is a locally constant, integer-valued function on $X$ (i.e. an element of $\ceH^0(X;\underline{\Z})$). In particular, if $X$ is connected, it is a constant map, i.e. an integer. 

\begin{definition}[Rank of a vector bundle]
Let $X$ be a connected topological space. The \textbf{rank} of a $\K$-vector bundle $p:E \lra X$ is the dimension of the $\K$-vector space $p^{-1}(\{x\})$, for any $x \in X$. It is denoted $\rk{E}$. A vector bundle of rank $1$ is called a \textbf{line bundle}.
\end{definition}

\begin{definition}[Homomorphism of vector bundles]
A \textbf{homomorphism}, or simply morphism, between two $\K$-vector bundles $p:E \lra X$ and $p': E' \lra X'$ is a pair $(u,f)$ of continuous maps $u:E \lra E'$ and $f:X \lra X'$ such that

\begin{enumerate}

\item the diagramme
$$\begin{CD}
E @>>u> E'\\
@VpVV  @Vp'VV \\
X @>>f> X'
\end{CD}
$$
\vskip5pt
\noindent is commutative,

\item for all $x\in X$, the map $$u_x: p^{-1}(\{x\}) \lra (p')^{-1}(\{f(x)\})$$ is $\K$-linear.

\end{enumerate}

\end{definition}

Topological vector bundles together with their homomorphisms form a category that we denote $\mathrm{Vect}^{\mathrm{top}}$. If $X$ is a fixed topological space, there is a category $\mathrm{Vect}^{\mathrm{top}}_X$ whose objects are topological vector bundles on $X$ and whose morphisms are defined as follows.

\begin{definition}[Homomorphisms of vector bundles on $X$]
Let $X$ be a fixed topological space and let $p:E \lra X$ and $p':E' \lra X$ be two $\K$-vector bundles on $X$. A \textbf{morphism of $\K$-vector bundles on $X$} is a continous map $u:E \lra E'$ such that 

\begin{enumerate}

\item the diagramme

$$
\xymatrix{
p^{-1}(U) \ar[rr]^{\simeq}_{\phi_U} \ar[dr]_p & & U \times \K^{r_U} \ar[dl]^{\mathrm{pr}_U} \\
& U 
}
$$

\noindent is commutative,

\item for all $x\in x$, the map $$u_x: p^{-1}(\{x\}) \lra (p')^{-1}(\{f(x)\})$$ is $\K$-linear.

\end{enumerate}

\end{definition}

As usual, an isomorphism is a homomorphism which admits an inverse homomorphism. A $\K$-vector bundle isomorphic to a product bundle is called a \textit{trivial bundle}.

When the base space $X$ is a smooth manifold, one defines smooth vector bundles on $X$ using smooth maps in place of continuous ones, and when $X$ is a complex analytic manifold manifold, one may accordingly define holomorphic vector bundles on $X$ (this last notion only makes sense,  of course, if the  field $\K$ in Definition \ref{def:vector_bundle} is assumed to be the field of complex numbers).

\subsubsection{Transition maps}

We shall henceforth assume that $X$ is connected. Let $E$ be a $\K$-vector bundle on $X$, and denote $r=\rk{E}$. If $(U,\phi_U)$ and $(V,\phi_V)$ are two overlapping charts in the sense that $U\cap V \neq \emptyset$, one gets a map

$$\phi_U \circ \phi_V^{-1} :
\begin{array}{ccc}
(U \cap V) \times \K^r & \lra & (U\cap V) \times \K^r \\
(x,v) & \lmt & (x, g_{UV}(x)\cdot v)
\end{array}
$$

\noindent where $g_{UV}: U\cap V \lra \Aut(\K^r) = \GL(r,\K)$. It satisfies, for any triple of open sets $(U,V,W)$,

\begin{equation}\label{cocycle_cond} 
g_{UV} g_{VW} = g_{UW}
\end{equation}

\noindent (the product on the left-hand side being the pointwise product of $\GL(r,\K)$-valued functions). Setting $U=V=W$, we obtain $(g_{UU})^2=g_{UU}$, so $$g_{UU}=I_r$$ (the constant map equal to $\mathrm{I}_r$). This in turn implies that $g_{UV}g_{VU} = g_{UU} = \mathrm{I}_r$, so $$g_{VU} = g_{UV}^{-1}$$ (the map taking $x$ to $(g_{UV}(x))^{-1} \in \GL(r,\K)$). The condition \eqref{cocycle_cond} is called a \textit{cocycle condition}. If $(U_i)_{i\in I}$ is a covering of $X$ by trivialising open sets, with associated local trivialisations $(\phi_i)_{i\in I}$, we get a family $$g_{ij}: U_i \cap U_j \lra \GL(r,\K)$$ of maps satisying condition \eqref{cocycle_cond}: the family $(g_{ij})_{(i,j)\in I\times I}$ is called a $\GL(r,\K)$-valued $1$\textit{-cocycle} subordinate to the open covering $(U_i)_{i\in I}$. It is completely determined by the transition maps $(\phi_i \circ \phi_j^{-1})_{(i,j)\in I\times I}$. Conversely, such a cocycle defines a topological $\K$-vector bundle of rank $r$

\begin{equation}\label{bundle_defined_by_gluing}
E:= \left(\bigcup_{i\in I}\ \{i\} \times U_i \times \K^r \right) \Big/ \sim\ ,
\end{equation}

\noindent where the equivalence relation $\sim$ identifies $(i,x,v)$ and $(j,y,w)$ if $y=x$ (in particular, $U_i\cap U_j \neq \emptyset$) and $w = g_{ij}(x) \cdot v$, the projection map $p:E \lra X$ being induced by the projections maps $\{i\}\times U_i \times \K^r \lra U_i \subset X$. In other words, a $\K$-vector bundle of rank $r$ on $X$ is a fibre bundle with typical fibre $\K^r$ and structure group $\GL(r,\K)$, acting on $\K^r$ by linear transformations (see for instance \cite{Steenrod}). Two vector bundles $E$ and $E'$ on $X$, represented by two cocycles $(g_{ij})_{(i,j)}$ and $(g'_{ij})_{(i,j)}$ subordinate to a same open covering $(U_i)_{i\in I}$ are isomorphic if and only if there exists a family $$u_i\, :\, U_i \lra \GL(r,\K)$$ of maps satisfying $$g'_{ij} = u_i g_{ij} u_j^{-1}.$$ Indeed, simply define $u_i$ in the following way $$U_i \times \K^r \underset{\phi_i^{-1}}{\lra} E|_{U_i} \underset{u}{\lra} E'|_{U_i} \underset{\phi_i'}{\lra} U_i\times \K^r \, ,$$ the map taking $(x,v)$ to $(x,u_i(x)\cdot v)$, and check that, for all $x\in U_i\cap U_j$, $u_i(x) = g'_{ij}(x) u_j(x) g_{ij}^{-1}(x)$. This defines an equivalence relation on the set of $\GL(r,\K)$-valued $1$-cocycles subordinate to a given open covering $\cU=(U_i)_{i\in I}$ of $X$. The set of equivalence classes for this relations is usually denoted $$\ceH^1_{\mathrm{top}}(\cU;\GL(r,\K)).$$ These sets form a direct system relative to the operation of passing from an open covering of $X$ to a finer one, and the direct limit is denoted $$\ceH^1_{\mathrm{top}}(X;\GL(r,\K)):= \varinjlim_{\cU} \ceH^1_{\mathrm{top}}(\cU;\GL(r,\K))$$ (see for instance \cite{Gunning_RS}). \textit{This set is the set of isomorphism classes of topological $\K$-vector bundles on $X$} (if $X$ is not connected, $\ceH^1(X;\GL(r,\K))$ is the disjoint union  $ \sqcup_{i=1}^k \ceH^1(X_i\, ; \GL(r,\K))$, where $\sqcup_{i=1}^k X_i$ is the disjoint union of connected components of $X$). If one considers smooth $1$-cocycles instead of continuous ones, one obtains a similar description of smooth vector bundles on $X$. Likewise, if $X$ is a Riemann surface, a holomorphic vector bundle of rank $r$ on $X$ is represented by a holomorphic $1$-cocycle $$g_{ij}: U_i \cap U_j \lra \GL(r,\C)$$ in the sense that all the components of $g_{ij}$ are holomorphic functions of one variable. An automorphism of a topological (resp. smooth, resp. holomorphic) vector bundle $E$ on $X$ represented by the cocycle $(g_{ij})_{(i,j)}$ may be represented by a family $(u_i\, :\, U_i \lra \GL(r,\K))_i$ of continuous (resp. smooth, resp. holomorphic) maps satisfying $u_i g_{ij} = g_{ij} u_j$ for all $(i,j)\in I\times I$ such that $U_i \cap U_j \neq \emptyset$. 

\subsubsection{Sections of a bundle}

Sections of a bundle are a generalisation of mappings between two spaces $X$ and $Y$ in the sense that a map from $X$ to $Y$ is a section of the product bundle $X\times Y\lra X$.

\begin{definition}[Sections of a vector bundle]
A (global) \textbf{section} of a topological $\K$-vector bundle $p:E\lra X$ is a continuous map $s:X\lra E$ such that $p\circ s=\Id_X$. The set $\Ga(E)$ of global sections of $E$ is an infinite-dimensional $\K$-vector space and a module over the ring of $\K$-valued functions on $X$.
\end{definition}

\noindent \textit{Local} sections of $E$ are sections $s_U:U\lra E|_U\simeq U\times \K^r$ of the vector bundle $E|_U$ where $U\subset X$ is an open subset. They may be seen as maps from $U$ to $\K^r$. If $(g_{UV})_{(U,V)}$ is a $1$-cocycle representing the vector bundle $E$, then a global section $s:X\lra E$ is the same as a collection $(s_U)_U$ of local sections subject to the condition $$s_U = g_{UV} s_V$$ for any pair $(U,V)$ of open subsets of $X$ satisfying $U\cap V\neq \emptyset$. Smooth (resp. holomorphic) sections of a smooth (resp. holomorphic) vector bundles are defined accordingly.

\begin{example}
A section of the vector bundle $TM\lra M$ is a \textit{vector field} on $M$. A section of $T^*M\lra M$ is a \textit{differential $1$-form} on $M$.
\end{example}

\subsection{Topological classification}

Evidently, if two $\K$-vector bundles on $X$ are isomorphic, they have the same rank (=dimension over $\K$ of the typical fibre of $E$). If $\K=\C$ and $X=\Si_g$ is a compact connected oriented \textit{surface}, isomorphism classes of \textit{topological} or \textit{smooth} vector bundles on $X$ are completely classified by a pair $(r,d)$ of integers, namely the rank and the \textit{degree} of a complex vector bundle. We give the following definition, which makes a free use of the notion of \textit{Chern class} of a complex vector bundle (see for instance \cite{Bott-Tu} or \cite{Hatcher}).

\begin{definition}[Degree]\label{def_of_degree}
Let $E$ be a complex vector bundle on a compact connected oriented surface $\Si_g$. The \textit{degree} of $E$ is by definition the integral of the first Chern class $c_1(E)\in H^2(\Si_g;\Z)$ of $E$: $$deg(E) := \int_{\Si_g} c_1(E) \in\Z.$$
\end{definition}

\noindent The degree of a vector bundle satisfies the following relations, which are often useful in computations: $$deg(E^*) = - deg(E)\quad \mathrm{and}\quad deg(E_1\otimes E_2) = deg(E_1)\,\rk(E_2) + \rk(E_1)\, deg(E_2).$$

\begin{theorem}[Topological classification of complex vector bundles on a curve]
Let $\Si_g$ be a compact connected Riemann surface, and let $E,E'$ be two topological (resp. smooth) complex vector bundles on $\Si_g$. Denote $r=\rk(E)$, $r'=\rk(E')$, $d=deg(E)$ and $d'=deg(E')$. Then $E$ and $E'$ are isomorphic as topological (resp. smooth) complex vector bundles on $\Si_g$ if and only if $r=r'$ and $d=d'$. Moreover, for any pair $(r,d)\in\Z_+\times \Z$, there exists a complex vector bundle of rank $r$ and degree $d$ on $\Si_g$.
\end{theorem}

\noindent We refer for instance to \cite{Thaddeus} for a proof of this theorem. In Section \ref{stable_bundles}, we shall be interested in the (much more involved) classification problem for \textit{holomorphic} vector bundles on a compact connected Riemann surface $\Si_g$, and the preceding result will be used to reduce it to the study of \textit{holomorphic structures} on a given smooth complex vector bundle of \textit{topological type} $(r,d)$.

\subsection{Dolbeault operators and the space of holomorphic structures}\label{Dolbeault_op}

In this subsection, we only consider smooth complex vector bundles over a fixed Riemann surface $\Si$ (although, most of the time, a similar, albeit slightly more complicated, statement holds for complex vector bundles over a higher-dimensional complex analytic manifold, see for instance \cite{Kobayashi} or \cite{Wells} for more on this topic).

\subsubsection{Smooth complex vector bundles and their sections}

A holomorphic structure on a topological complex vector bundle is by definition a (maximal) holomorphic atlas on it (local trivialisations with holomorphic transition maps). Such a holomorphic structure defines (up to conjugation by an automorphism of the bundle) a remarkable object on the underlying \textit{smooth} complex vector bundle: a Dolbeault operator.

In what follows, we will denote $E$ a smooth complex vector bundle on $\Si$. When $E$ is endowed with a holomorphic structure, we will designate by $\cE$ the resulting holomorphic vector bundle. We denote $\Om^0(\Si;E)= \Ga(E)$ the complex vector space of smooth sections of $E$, and $\Om^k(\Si;E)$ the complex vector space of $E$-valued, smooth, $\R$-linear $k$-forms on $\Si$. For any $k\geq 0$, $\Om^k(\Si;E)$ is also a module over the ring $\Om^0(\Si;\C)=C^\infty(\Si;\C)$ of $\C$-valued smooth functions on $\Si$. It is important to stress that, for $k\geq 1$, $\Om^k(\Si;E)$ is the space of smooth sections of the complex vector bundle $\wedge^k (T^*\Si) \otimes_{\R} E$. It is a complex vector space because the fibres of $E$ are complex vector spaces, but a single element $\w\in\Om^k(\Si;E)$, when evaluated at a point $x\in\Si$, defines an \textit{$\R$-linear} map $$\w_x: T_x\Si \wedge \cdots \wedge T_x\Si \lra E_x.$$ In particular, we do not restrict our attention to $\C$-linear such maps. Instead, we write, for instance if $k=1$, $$\Om^1(\Si;E) = \Om^{1,0}(\Si;E) \oplus \Om^{0,1}(\Si;E)$$ where $\Om^{1,0}(\Si;E)$ is the complex vector space of $\C$-linear $1$-forms $$\w: T\Si \lra E$$ and $\Om^{0,1}(\Si;E)$ is the complex vector space of $\C$-antilinear such forms. We recall that $\Om^k(\Si;\C)=\Om^k(\Si;E)$ for $E=\Si\times \C$, so the remark above is in particular valid for $\Om^1(\Si;\C)$. For $k > 1$ (and over a complex analytic manifold $M$ of arbitrary dimension), we would have a decomposition $$\Om^k(M;E) = \Om^{k,0}(M;E) \oplus \Om^{k-1,1}(M;E) \oplus \cdots \oplus \Om^{1,k-1}(M;E) \oplus \Om^{0,k}(M;E)$$ where $\Om^{p,q}(M;E)$ is the space of $\R$-linear $(p+q)$-forms $$\w: TM \wedge \cdots \wedge TM \lra E$$ which are $\C$-linear in $p$ arguments and $\C$-antilinear in $q$ arguments. A consequence of these decompositions is that the de Rham operator $d$ on $\Si$ splits into $$d= d^{1,0} \oplus d^{0,1}: \Om^0(\Si;\C) \lra \Om^1(\Si;\C)= \Om^{1,0}(\Si;\C)  \oplus \Om^{0,1}(\Si;\C).$$ That is, the exterior derivative $df$ of a smooth function $f:\Si \lra \C$ splits into a $\C$-linear part $d^{1,0}f$ (also denoted $\del{f}$) and a $\C$-antilinear part $d^{0,1}f$ (also denoted $\delb{f}$).

\begin{lemma}\label{Cauchy-Riemann_op}
A smooth function $f:\Si \lra \C$ is holomorphic if and only if $\delb{f}=0$.
\end{lemma}

\begin{proof}
It is a question of a purely local nature. In a holomorphic chart  $z=x+iy$ of $\Si$, one has $df= \del{f} dz +\delb{f}d\overline{z}$ and $\delb{f} = \frac{1}{2}(\frac{\del{f}}{\del{x}} + i \frac{\del{f}}{\del{y}})$. Write now $f=P+iQ$ with $P$ and $Q$ real-valued. Then $\delb{f}=0$ if and only if $\frac{\del{P}}{\del{x}} = \frac{\del{Q}}{\del{y}}$ and $\frac{\del{Q}}{\del{x}} = - \frac{\del{P}}{\del{y}}$. These are the Cauchy-Riemann equations: they mean that the Jacobian matrix of $f$ at any given point is a similitude matrix, i.e. a complex number, which amounts to saying that $f$ is holomorphic.
\end{proof}
 
\begin{definition}[Cauchy-Riemann operator]
The operator $$\delb:\Om^0(\Si,\C) \lra \Om^{0,1}(\Si;\C)$$ taking a smooth function to the $\C$-antilinear part of its derivative is called the \textbf{Cauchy-Riemann operator} of the Riemann surface $\Si$.
\end{definition} 

\noindent Observe that the Cauchy-Riemann operator satisfies the Leibniz rule $$\delb(fg) = (\delb{f})g + f(\delb{g}).$$
We would like to have a similar characterisation for the holomorphic \textit{sections} of an arbitrary holomorphic bundle $\cE$ (not just $\Si\times\C$). The problem is that there is no canonically defined operator $$D:\Om^0(\Si;E) \lra \Om^1(\Si;E)$$ which would play the role of the de Rham operator, so we first need to define these. Recall that the de Rham operator satisfies the Leibniz rule $$d(fg) = (df)g + f(dg).$$ The next object, called a (linear) connection, gives a way to differentiate sections of a vector bundle $E$ \textit{covariantly} (in such a way that the resulting object is an $E$-valued $1$-form, thus generalising the de Rham operator). To give a presentation of connections of a broader interest, we temporarily move back to manifolds more general than Riemann surfaces. The next definition is a bit abstract but designed to incorporate the case, for instance, of smooth complex vector bundles over real smooth manifolds. 

\begin{definition}[Linear connection]
Let $M$ be a smooth manifold over $\K=\R$ or $\C$. A (linear) \textbf{connection} on a smooth vector bundle $E\lra M$ is a $\K$-linear map $$D: \Om^0(M;E) \lra \Om^1(M;E)$$ satisfying the following Leibniz rule $$D(fs) = (df)s + f(Ds)$$ for all $f\in C^\infty(M;\K)$ and all $s\in \Om^0(M;E)$, where $d$ is the de Rham operator on $M$.
\end{definition}

\noindent $(df)s$ is the element of $\Om^1(M;E)$ which, when evaluated at $x\in M$, is the $\R$-linear (\textit{even, and this is crucial, if} $\K=\C$) map

$$
\begin{array}{ccc}
T_x M  & \lra & E_x \\
v & \lmt & \underbrace{(df)_x(v)}_{\in \K} \cdot \underbrace{s(x)}_{\in E_x}\ .
\end{array}
$$

\noindent $fDs$ is the element of $\Om^1(M;E)$ which, when evaluated at $x\in M$, is the $\R$-linear map $f(x)(Ds)_x$. Given a section $s\in\Om^0(M;E)$, the $E$-valued $1$-form $Ds$ is called the \textit{covariant derivative} of $s$. We observe that the product bundle $U\times \K^r$ always admits a distinguished connection, called the product connection: since elements of $\Om^1(U;U\times\K^r) \simeq \Om^1(U;\K^r)$ are $r$-uplets of $k$-forms on $U$, one may define $D=d\oplus \cdots \oplus d$, the de Rham operator repeated $r$ times. One can moreover show that any convex combination of linear connections is a linear connection. It is then a simple consequence of the existence of partitions of unity that any smooth vector bundle admits a connection (see Exercise \ref{convex_comb_of_conn}). As we now show, the space of all connections is an affine space.

\begin{proposition}
The space of all linear connections on a smooth $\K$-vector bundle is an affine space, whose group of translations is the vector space $\Om^1(\Si;\End(E))$.
\end{proposition}

\begin{proof}
It suffices to show that the difference $D_1-D_2$ of two linear connections defines an element of $\Om^1(M;E)$. One has, for all $f\in C^\infty(M;\K)$ and all $s\in\Om^0(M;E)$,

\begin{eqnarray*}
(D_1-D_2)(fs) & = & D_1(fs) - D_2(fs) \\
& = & (df)s + f(D_1 s) -(df)s - f(D_2 s) \\
& = & f(D_1-D_2)s.
\end{eqnarray*}

\noindent So $D_1-D_2$ is a $C^\infty(M;\K)$-linear map from $\Om^0(M;E)$ to $\Om^1(M;E)$. This is the same as an $\End(E)$-valued $1$-form on $M$.
\end{proof}

We postpone the exposition of further generalities on linear connections (curvature and the like) to Subsection \ref{unitary_conn}, where they will be presented for a special case of linear connections called unitary connections (but seen to hold in greater generality), and we go back to generalised Cauchy-Riemann generators. 

A linear connection on a smooth complex vector bundle $E\lra \Si$ over a Riemann surface splits into $$D= D^{1,0} \oplus D^{0,1}\, :\, \Om^0(\Si;E) \lra \Om^1(\Si;E) = \Om^{1,0}(\Si;E) \oplus \Om^{0,1}(\Si;E).$$

\begin{lemma}\label{(0,1)-part}
Let $D$ be a linear connection on $E\lra\Si $. The operator $$D^{0,1}: \Om^0(\Si;E) \lra \Om^{0,1}(\Si;E)$$ taking a section of $E$ to the $\C$-antilinear part of its covariant derivative is $\C$-linear and satisfies the following Leibniz rule $$D^{0,1}(fs) = (\delb{f})s + f (D^{0,1}s),$$ where $\delb$ is the Cauchy-Riemann operator on $\Si$.
\end{lemma}

\begin{proof}
$D^{0,1}$ is obviously additive. Moreover,

\begin{eqnarray*}
D(fs) & = & (df) s + f(Ds) \\
& = & (\del{f} )s + f(D^{1,0}s) + (\delb{f})s + f(D^{0,1}s)
\end{eqnarray*}

\noindent so the $\C$-antilinear part of $D(fs)$ is $(\delb{f})s+f(D^{0,1}s)$.
\end{proof}

\noindent This motivates the following definition.

\begin{definition}[Dolbeault operator]
A Dolbeault operator on a smooth complex vector bundle $E\lra \Si$ over a Riemann surface is a $\C$-linear map $$D'':\Om^0(\Si;E) \lra \Om^{0,1}(\Si;E)$$ satisfying the following Leibniz rule: for all $f\in C^\infty(\Si;\C)$ and all $s\in \Om^0(\Si;E)$, $$D''(fs) = (\delb{f})s + f(D''s)$$ where $\delb$ is the Cauchy-Riemann operator of $\Si$.
\end{definition}

A Dolbeault operator is also called a $(0,1)$-connection. As in the case of connections, any smooth complex vector bundle over a complex base space admits a Dolbeault operator, and the space $\Dol(E)$ of all Dolbeault operators on $E$ is an affine space, whose group of translations is the vector space $\Om^{0,1}(\Si;\End(E))$. We now show that, given a \textit{holomorphic} vector bundle $\cE$ on $\Si$, there is a Dolbeault operator on the underlying smooth vector bundle $E$, whose kernel consists exactly of the holomorphic sections of $\cE$ (much like $\ker\delb$ consists of the holomorphic functions on $\Si$, see Lemma \ref{Cauchy-Riemann_op}). We first observe that $\cG_E$, the group\begin{footnote}{One may observe that there is a group bundle $\GL(E)$ on $\Si$, whose typical fibre is $\GL(r;\C)$ and whose structure group is $\Ad{\GL(r;\C)}$, such that $\cG_E=\Ga(\GL(E))$.}\end{footnote} of all complex linear bundle automorphisms of $E$, acts on $\Dol(E)$ in the following way: if $u\in \cG_E$ and $D''\in\Dol(E)$, 

\begin{equation}\label{action_on_Dolbeault_op}
(u\cdot D'')(s) := u\big( D''(u^{-1}s)\big)
\end{equation}

\noindent is a Dolbeault operator on $E$ (Exercise \ref{action_on_Dol(E)}), and, if $v$ is another automorphism of $E$, $$(uv) \cdot D'' = u \cdot (v\cdot D'')$$ so we indeed have a group action. Moreover, $D''$ is a local operator in the following sense: if $s_U$ is a \textit{local} smooth section of $E$, we can define a local $E$-valued $1$-form $D''s_U$ using bump functions on $U$ (see the proof below). In particular, local solutions to the equation $D''s=0$ form a sheaf on $\Si$.

\begin{proposition}
Let $E$ be a smooth complex vector bundle on $\Si$. Given a holomorphic structure on $E$, denote $\cE$ the resulting holomorphic vector bundle. Then, there exists a unique $\cG_E$-orbit of Dolbeault operators on $E$ such that, for any $D''$ in that orbit, local holomorphic sections of $\cE$ are in bijection with local solutions to the equation $D''s=0$.
\end{proposition}

\begin{proof}
Let $(g_{ij})_{(i,j)}$ be a \textit{holomorphic} $1$-cocycle of transition maps on $E$. Let $s$ be a smooth global section of $E$, and denote $s_i$ the section $s$ read in the local chart $(U_i,\phi_i)$. Then $s_i= g_{ij} s_j$ as maps from $U_i \cap U_j \to \C^r$, so $$\delb{s_i} = \delb{(g_{ij}s_j)} = (\delb{g_{ij}})s_j + g_{ij}(\delb{s_j}) = g_{ij}(\delb{s_j}),$$ since $g_{ij}$ is holomorphic. This defines, for any $s\in\Om^0(\Si;E)$, a global, $E$-valued $(0,1)$-form $D''s$ on $\Si$ (such that $(D''s)_i = \delb{s_i}$). The Leibniz rule for the operator thus defined follows from the Leibniz rule for the the local operator $\delb$. Let us now identify the local solutions to the equation $D''s=0$. If $\si$ is a smooth local section of $E$ over an open subset $U$ of $\Si$, let $f$ be a smooth bump function whose support is contained in $U$. Then, by definition, there exists an open set $V\subset U$ on which $f$ is identically $1$. We may assume that $V$ is trivialising for $E$. Then $\sigma|_V:V\to\C^r$ is a smooth local section of $E$ over $V$, and it is holomorphic in $V$ if and only if $\delb{(\si|_V)}=0$, or equivalently, $(D''\si)_V=0$ (observe that this last equation makes sense because we can extend $\si$ \textit{smoothly} to $\Si$ using the bump function, and since $f\equiv 1$ in $V$, $(D''\si)_V$ does not depend on the extension). Using different pairs $(V,f)$, we see that $\si$ is holomorphic in $U$ if and only if it is a local solution to the equation $D''s=0$. Evidently, two isomorphic holomorphic structures on $E$ determine conjugate Dolbeault operators.
\end{proof}

\noindent So we have an injective map $$\{\textrm{holomorphic\ structures\ on}\ E\}\, /\,  \textrm{isomorphism} \lra \Dol(E) /\, \cG_E\, .$$ It is a remarkable fact that the image of this map can be entirely described, and that it is in fact a surjective map when the base complex analytic manifold has complex dimension $1$. The problem of determining whether a Dolbeault operator comes from a holomorphic structure is a typical \textit{integrability} question, similar to knowing whether a linear connection comes from a linear representation of the fundamental group of $M$ at a given basepoint. The integrability conditions, too, are very similar, and we refer to \cite{DK} (Chapter 2, Section 2) for an illuminating parallel discussion of the two questions, as well as the proof of the following integrability theorem.

\begin{theorem}[The Newlander-Nirenberg Theorem in complex dimension one]
Let $E\lra \Si$ be a smooth complex vector bundle on a Riemann surface, and let $D''$ be a Dolbeault operator on $E$. Then there exists a unique holomorphic structure on $E$ such that such that the local holomorphic sections of $\cE$ are in bijection with smooth local solutions to the equatio $D''=0$.
\end{theorem}

The proof is a question of showing that the sheaf of local solutions to the equation $D''s=0$ is a locally free sheaf of rank $r=\rk{E}$ over the sheaf of holomorphic functions of $\Si$. It is then easy to check that two $\cG_E$-conjugate Dolbeault operators determine isomorphic holomorphic structures, since their kernels are conjugate in $\Om^0(\Si;E)$. Therefore, over a Riemann surface $\Si$, there is a bijection between the set of isomorphism classes of holomorphic structures on $E$ and $$\Dol(E) / \,\cG_E\, .$$ We refer to \cite{VLP}, Expos\'e 1, for an explanation of why any natural topology of this space is not Hausdorff. We conclude the present subsection by one further remark on Dolbeault operators, namely that a Dolbeault operator $D''$ on $E$ induces a Dolbeault operator $D''_E$ on $\End(E)$. First, note that, for all $k\geq 0$, $$\Om^k(\Si;\End(E)) =\Ga( \wedge^k T^*\Si \otimes_\R \End(E)) = \Hom(E;\wedge^k T^*\Si \otimes_\R E)\, .$$ So, given $u\in\Om^0(\Si;\End(E))$, we need only specify $(D''_Eu)(s)\in\Om^1(\Si;E)$ for all $s\in\Om^0(\Si;E)$ in order to completely determine $D''_Eu\in\Om^1(\Si;\End(E))$. Moreover, because $u(s)$ is locally a product between a matrix and a column vector, we want the  would-be operator $D''_E$ on $\End(E)$ to satisfy, for all $u\in\Om^0(\Si;\End(E))$ and all $s\in \Om^0(\Si;E)$, the generalised Leibniz identity

\begin{equation}\label{Leibniz_endo}
D''\big(u(s)\big) = (D''_Eu)(s) + u(D''s)
\end{equation}

\noindent so we \textit{define} $$(D''_Eu)(s) := D''\big(u(s)\big) - u (D''s).$$ Evidently, $D''_E$ is $\C$-linear in $u$ as a map from $\Om^0(\Si;\End(E))$ to $\Om^1(\Si;\End(E))$, and, if $f\in C^\infty(\Si;\C)$, one has

\begin{eqnarray*}
\big(D''_E(fu)\big) (s) & = & D''\big(f u(s)\big) - (fu) D''s \\
& = & (\delb{f})\, u(s)+ fD''\big(u(s)\big) - f \big(u(D''s)\big) \\
& = & [ (\delb{f}) u + f(D''_Eu) ] (s)
\end{eqnarray*}

\noindent so $D''_E$ is indeed a Dolbeault operator on $\End(E)$. In practice, it is simply denoted $D''$, which, if anything, makes \eqref{Leibniz_endo} more transparent and easy to remember. As a consequence of the Leibniz identity \eqref{Leibniz_endo}, one can modify the way the $\cG_E$-action on $\Dol(E)$ is written:

\begin{eqnarray*}
(u\cdot D'')(s) & = & u\big(D''(u^{-1}s)\big)\\
& = & u\big((D''(u^{-1}))s + u^{-1}(D''s)\big) \\
& = & u(-u^{-1}(D''u)u^{-1}s + u^{-1}D''s) \\
& = & D''s - (D''u)u^{-1}s
\end{eqnarray*}

\noindent so

\begin{equation}\label{gauge_action_Dolbeault}
u\cdot D'' = D'' - (D''u)u^{-1}\, ,
\end{equation}

\noindent which is the way the $\cG_E$-action on $\Dol(E)$ is usually written when performing explicit computations (we refer to Exercise \ref{derivative_of_inverse} for the computation of $D''(u^{-1})$ used in the above).\vskip10pt

\subsection{Exercises}\hfill

\begin{multicols}{2}

\begin{exercise}
Show that the quotient topological space defined in \eqref{bundle_defined_by_gluing} is a vector bundle of rank $r$ on $X$.
\end{exercise}

\begin{exercise}\label{Mobius}
Show that the M\"obius bundle $\cM\lra S^1$ is isomorphic to the canonical bundle $E_{\mathrm{can}} \lra \RP^1$ (see Example \ref{examples_of_vb}).
\end{exercise}

\begin{exercise}\label{product_bundle}
Show that a vector bundle of rank $r$, $p:E\lra X$, say, is isomorphic to the product bundle $X\times \K^r$ if and only f there exist $r$ global sections $s_1, \cdots, s_r\in\Ga(X;E)$ such that, for all $x\in X$, $(s_1(x),\cdots, s_r(x))$ is a basis of $E_x$ over $\K$.
\end{exercise}

\begin{exercise}\label{convex_comb_of_conn}
\textbf{a.} Let $(D_i)_{1 \leq i\leq n}$ be $n$ linear connections on a vector bundle $E\lra M$, and let $(\lambda_i)_{1\leq i\leq n}$ be $n$ non-negative real numbers satisfying $\sum_{i=1}^n \lambda_i =1$. Show that the convex combination $D=\sum_{i=1}^n \lambda_i D_i$ is a linear connection on $E$.\\
\textbf{b.} Let $(U_i)_{i\in I}$ be a covering of $M$ by trivialising open sets for $E$, and let $D_i$ be the product connection on $E|_{U_i}$. Let $(f_i)_{i\in I}$ be a partition of unity subordinate to $(U_i)_{i\in I}$. Show that $D:=\sum_{i\in I}D_i$ is a well-defined map from $\Om^0(M;E)$ to $\Om^1(M;E)$, and that it is a linear connection on $E$.
\end{exercise}

\begin{exercise}\label{action_on_Dol(E)}
Let $E$ be a smooth complex vector bundle and let $D''$ be a Dolbeault operator on $E$. Let $u$ be an automorphism of $E$. Define $u\cdot D''$ by $$(u\cdot D'')(s) = u\big(D''(u^{-1}s)\big)$$ on sections of $E$.\\ \textbf{a.} Show that $u\cdot D''$ is a Dolbeault operator on $E$.\\ \textbf{b.} Show that this defines an action of the group of automorphisms of $E$ on the set of Dolbeault operators.
\end{exercise}

\begin{exercise}\label{derivative_of_inverse}
Let $D''$ be a Dolbeault operator on a smooth complex vector bundle $E$, and let $u$ be an automorphism of $E$. Show that $$D''(u^{-1}) = - u^{-1}(D''u)u^{-1}.$$
\end{exercise}

\end{multicols}

\section{Holomorphic structures and unitary connections}\label{unitary_connections}

In this section, we study the space of holomorphic structures on a smooth complex vector bundle $E$ over a Riemann surface $\Si$ \textit{in the additional presence of a Hermitian metric $h$ on $E$}. This has the effect of replacing the space of Dolbeault operators by another space of differential operators: the space of unitary connections on $(E,h)$. This new affine space turns out to have a natural structure of infinite-dimensional K\"ahler manifold. Moreover, the action of the group of unitary transformations of $(E,h)$ on the space of unitary connections is a Hamiltonian action, and this geometric point of view, initiated by Atiyah and Bott in \cite{AB}, will be key to understanding Donaldson's Theorem in Subsection \ref{section_on_Donaldson_thm}.

\subsection{Hermitian metrics and unitary connections}\label{unitary_conn}

\begin{definition}[Hermitian metric]
Let $E$ be a smooth complex vector bundle on a smooth manifold $M$. A \textbf{Hermitian metric} $h$ on $E$ is a family $(h_x)_{x\in X}$ of maps $$h_x: E_x \times E_x \lra \C$$ such that

\begin{enumerate}
\item $\forall(v,w_1,w_2)\in E_x \times E_x\times E_x$, $$h(v, w_1+w_2)= h(v,w_1) + h(v,w_2),$$
\item $\forall (v,w)\in E_x \times E_x$, $\forall \lambda\in\C$, $$h(v,\lambda w) = \lambda h(v,w),$$
\item $\forall (v,w)\in E_x \times E_x$, $$h(w,v) = \overline{h(v,w)},$$
\item $\forall v\in E_x\setminus \{0\}$, $$h(v,v)>0,$$
\item for any pair $(s,s')$ of smooth sections of $E$, the function $$h(s,s'): M \lra \C$$ is smooth.
\end{enumerate}

\end{definition}

In other words, $h$ is a smooth family of Hermitian products on the fibres of $E$. A smooth complex vector bundle with a Hermitian metric is called a \textbf{smooth Hermitian vector bundle}.

\begin{definition}
A \textbf{unitary transformation} of $(E,h)$ is an automorphism $u$ of $E$ satisfying, for any pair $(s,s')$ of smooth sections of $E$, $$h\big(u(s),u(s')\big) = h(s,s').$$ In other words, a unitary transformation is fibrewise an isometry. The group $\cG_h$ of unitary transformations of $(E,h)$ is called the (unitary) \textbf{gauge group}. There is a group bundle $\U(E,h)$, whose typical fibre is $\U(r)$ and whose structure group is $\Ad{\U(r)}$, such that $\cG_h=\Ga(\U(E,h))$.\\ A \textbf{Hermitian transformation} is an endomorphism $u$ of $E$ satisfying $$h\big(u(s),s'\big) = h\big(s,u(s')\big).$$ An \textbf{anti-Hermitian transformation} is an endomorphism $u$ of $E$ satisfying $$h\big(u(s),s'\big) = -h\big(s,u(s')\big).$$ The Lie algebra bundle whose sections are anti-Hermitian endomorphisms of $E$ is denoted $\frak{u}(E,h)$. Its typical fibre is the Lie algebra $\frak{u}(r)=Lie(\U(r))$ and its structure group is $\Ad{\U(r)}$. As one might expect, $\Ga(\frak{u}(E,h)) = \Om^0(M;\frak{u}(E,h))$ is actually the Lie algebra of $\Ga(\U(E,h))=\cG_h$.
\end{definition}

\begin{proposition}[Reduction of structure group]\label{reduction_of_structure_group}
Let $(E\lra M)$ be a smooth complex vector bundle. Given a Hermitian metric $h$ on $E$, there exists a $\U(r)$-valued $1$-cocycle $$g_{ij}: U_i \cap U_j \lra \U(r) \subset \GL(r,\C)$$ representing $E$. Two such cocycles differ by an $\U(r)$-valued $0$-cocycle. Conversely, an atlas of $E$ whose transition maps are given by a unitary $1$-cocycle determines a Hermitian metric on $E$.
\end{proposition}

More generally, if $H$ is a subgroup of $\GL(r,\C)$ and a vector bundle $E$ can be represented by an $H$-valued $1$-cocycle whose class modulo $H$-valued $0$-cocycles is uniquely defined, one says that the structure group of $E$ has been \textit{reduced} to $H$. The proposition above says that a Hermitian metric is equivalent to a reduction of the structure group $\GL(r,\C)$ of a complex rank $r$ vector bundle to the maximal compact subgroup $\U(r)$. In the general theory of fibre bundles, the existence of such a reduction is usually deduced from the fact that the homogeneous space $\GL(r,\C)/\U(r)$ (the space of Hermitian inner products on $\C^r$) is contractible (see for instance \cite{Steenrod}).

\begin{proof}[Proof of Proposition \ref{reduction_of_structure_group}]
Let $h$ be a Hermitian metric on $E$. Using the Gram-Schmidt process, one can obtain an $h$-unitary local frame out of any given local frame of $E$, hereby identifying $E|_U$ with $U\times \C^r$ where $\C^r$ is endowed with its canonical Hermitian inner product. The transition functions of such an atlas have an associated $1$-cocycle of transition maps which preserves the Hermitian product and is therefore $\U(r)$-valued. A different choice of unitary frames leads to a $\U(r)$-equivalent $1$-cocycle. Conversely, given such a $1$-cocycle, the Hermitian products obtained on the fibres of $E|_U$ and $E|_V$ respectively via the identifications with $U\times\C^r$ and $V\times \C^r$ coincide over $U\cap V$. 
\end{proof}

Since the structure group of the bundle has changed, it makes sense to ask whether there is a notion of connection which is compatible with this smaller structure group. This is usually better expressed in the language of principal bundles, but we shall not need this point of view in these notes (see for instance \cite{Kobayashi_Nomizu}).

\begin{definition}[Unitary connection]
Let $(E,h)$ be a smooth Hermitian vector bundle on a manifold $M$. A linear connection $$D: \Om^0(M;E) \lra \Om^1(M;E)$$ on $E$ is called \textbf{unitary} if, for any pair $(s,s')$ of smooth sections of $E$, one has $$d\big(h(s,s')\big) = h(Ds,s) + h(s,Ds')\, .$$
\end{definition}

The same standard arguments as in the case of linear connections and Dolbeault operators show that a Hermitian vector bundle always admits a unitary connection (locally, the product connection satisfies the unitarity condition), and that the space $\cA(E,h)$ of all unitary connections on $(E,h)$ is an affine space, whose group of translations is the vector space $\Om^1(M;\frak{u}(E,h))$ of $\frak{u}(E,h)$-valued $1$-forms on $M$.

Given a $k$-form $\alpha\in\Om^k(M;E)$ and a local trivialisation $(U,\phi_U)$ of $E$, let us denote 
$\alpha_U\in\Om^k(U;\C^r)$ the $k$-form obtained from reading $\alpha|_{U}$ in the local trivialisation $\phi_U: E|_{U} \overset{\simeq}{\lra} U\times\C^r$. Then, if $(g_{UV})_{(U,V)}$ is a $1$-cocycle of transition maps for $E$, one has $\alpha_U=g_{UV}\alpha_V$. Moreover, any linear connection is locally of the form $$(Ds)_U = d(s_U) + A_Us_U$$ where $A_U\in\Om^1(U;\frak{gl}(r,\C))$ is a family of matrix-valued $1$-forms defined on trivialising open sets by $A_Us_U = (Ds)_U - d(s_U)$, $d$ being the product de Rham operator on $\Om^1(U;\C^r)$, and subject to the condition, for all $s\in\Om^0(M;E)$,

\begin{eqnarray*}
(Ds)_U & = & g_{UV}(Ds)_V \\ 
& = & g_{UV} (ds_V + A_Vs_V) \\
& = & g_{UV} \big( d(g_{UV}^{-1}s_U) + A_V(g_{UV}^{-1}s_U)\big) \\
& = & g_{UV}\big(d(g_{UV}^{-1})s_U  + g_{UV}^{-1}ds_U\big) + g_{UV}A_Vg_{UV}^{-1}\, s_U \\
& = & ds_U + \big(g_{UV}A_Vg_{UV}^{-1} - (dg_{UV})g_{UV}^{-1}\big)s_U
\end{eqnarray*}

\noindent so

\begin{equation}\label{connection_gluing}
A_U = g_{UV}A_Vg_{UV}^{-1} - (dg_{UV})g_{UV}^{-1}\, .
\end{equation}

The family $(A_U)_U$ subject to Condition \eqref{connection_gluing} above is sometimes called the connection form, even though it is \textit{not} a global differential form on $M$. The connection determined by such a family is denoted $d_A$, or even simply $A$. Let us now analyse what it means for $d_A$ to be unitary. Given a pair $(s,s')$ of smooth sections of $E$, and a local chart $(U,\phi_U)$ of $E$, one has, on the one hand, $$d\big(h(s_U,s'_U)\big) = h(ds_U,s'_U) + h(s_U,ds'_U)$$ and, on the other hand, 

\begin{eqnarray*}
& & h\big((Ds)_U,s'_U\big) + h\big(s_U, (Ds')_U\big) \\
& = & h(ds_U, s'_U) + h(A_Us_U,s'_U) + h(s_U,ds'_U) + h(s_U,A_Us'_U)
\end{eqnarray*} 

\noindent so $d_A$ is unitary if and only if $$h(A_Us_U,s'_U) + h(s_U,A_Us'_U)=0$$ which means that the $1$-form $A_U$ is in fact $\frak{u}(r)$-valued. Conversely, if $(E,h)$ is represented by a unitary cocycle $(g_{UV}: U \cap V \lra \U(r))_{(U,V)}$ (Proposition \ref{reduction_of_structure_group}) and $(A_U)_U$ is a family of $\frak{u}(r)$-valued $1$-forms satisfying condition \eqref{connection_gluing}, then there is a unique unitary connection $d_A$ on $(E,h)$ such that, for all $s\in\Om^0(M;E)$, one has $(d_As)_U = ds_U + A_Us_U$ on each $U$.

Just like any linear connection on a smooth complex vector bundle over a complex manifold (Lemma \ref{(0,1)-part}), a unitary connection $$d_A:\Om^0(M;E) \lra \Om^1(M;E) = \Om^{1,0}(M;E) \oplus \Om^{0,1}(M;E)$$ splits into $d_A = d_A^{\, 1,0} \oplus d_A^{\, 0,1}$, where $d_A^{\, 1,0}$ takes a section $s$ of $E$ to the $\C$-linear part of its covariant derivative, and $d_A^{\, 0,1}$ takes $s$ to the $\C$-antilinear part of $d_As$. In particular, $$d_A^{\, 0,1}\, :\, \Om^1(M;E) \lra \Om^{0,1}(M;E)$$ is a Dolbeault operator. So, if $M=\Si$ is a Riemann surface, then, by the Newlander-Nirenberg Theorem, $d_A^{\, 0,1}$ determines a holomorphic structure on $E$. The next proposition shows what we gain by working in the presence of a Hermitian metric on $E$: a Dolbeault operator $D''$ on $E\lra M$ may be the $(0,1)$-part of various, non-equivalent linear connections, but it is the $(0,1)$-part of a \textit{unique} unitary connection.

\begin{proposition}
Let $(E,h)$ be a smooth Hermitian vector bundle on a complex manifold $M$, and let $$D'': \Om^0(M;E) \lra \Om^{0,1}(M;E)$$ be a Dolbeault operator on $E$. Then there exists a unique unitary connection $$d_A: \Om^0(M;E)\lra \Om^1(M;E)=\Om^{1,0}(M;E) \oplus \Om^{0,1}(M;E)$$ such that $d_A^{\, 0,1}=D''$.
\end{proposition}

\begin{proof}
Like for many other results in these notes, the proof essentially boils down to linear algebra. Let $(g_{UV})_{(U,V)}$ be a unitary $1$-cocycle representing $(E,h)$. The Dolbeault operator $D''$ is locally of the form $$(D''s)_U = \delb{s_U} + B_U s_U$$ where $B_U\in\Om^{0,1}(U;\frak{gl}(r,\C))$ and $\delb$ is the product Cauchy-Riemann operator on $\Om^0(U;\C^r)$, and where the family $(B_U)_{U}$ satisfies

\begin{equation}\label{Dolbeault_gluing}
B_U = g_{UV} B_V g_{UV}^{-1} - (\delb{g_{UV}})g_{UV}^{-1}
\end{equation}

\noindent (this does not require $g_{UV}$ to be unitary). We then have an isomorphism of real vector spaces

$$\begin{array}{ccc}
\Om^{0,1}(U;\frak{gl}(r,\C)) & \lra & \Om^1(U;\frak{u}(r)) \\
B_U & \lmt & A_U:= B_U -B_U^*
\end{array}$$

\noindent where $B_U^* = \overline{B_U}^t$ is the adjoint of $B_U$, the converse map being $$A_U \lmt A_U^{0,1} = \frac{A_U(\,\cdot\,) + i A_U(i\,\cdot\,)}{2}\cdot$$ One may observe here that $$A_U^{1,0} = \frac{A_U(\,\cdot\,) - i A_U(i\,\cdot\,)}{2} = -B_U^*\, .$$ Moreover, as $g_{UV}$ is unitary, $g_{UV}^* = g_{UV}^{-1}$ and therefore

\begin{eqnarray*}
A_U & = & B_U - B_U^* \\
& = & g_{UV} B_V g_{UV}^{-1} - (\delb{g_{UV}})g_{UV}^{-1} - g_{UV}B_V^*g_{UV}^{-1} + g_{UV}(\del{(g_{UV}^*)}) \\
& = & g_{UV}(B_V - B_V^*) g_{UV}^{-1} - (\del{g_{UV}} + \delb{g_{UV}})g_{UV}^{-1} \\
& = & g_{UV} A_V g_{UV}^{-1} - (dg_{UV})g_{UV}^{-1}
\end{eqnarray*}

\noindent so the family $(A_U)_U$ is a unitary connection on $(E,h)$, and $d_A^{\, 0,1} = (B_U)_U = D''$. Conversely, if $(A_U)_U$ is a unitary connection on $(E,h)$ such that $A_U^{0,1}=B_U$ for all $U$, then $A_U^{1,0} = -(A_U^{0,1})^* = -B_U^*$, so such a unitary connection is unique: $A_U = A_U^{1,0} + A_U^{0,1}=- B_U^* + B_U$.
\end{proof}

Observe that the family $(B_U)_U$ satisfying Condition \eqref{Dolbeault_gluing} completely determines the Dolbeault operator $D''$, which therefore could be denoted $\delb_B$, or even simply $B$.

\begin{corollary}
Let $\Si$ be a Riemann surface, and let $E$ be a smooth complex vector bundle on $\Si$. Then the choice of a Hermitian metric $h$ on $E$ determines an isomorphism of affine spaces

$$\begin{array}{ccc}
\cA(E,h) & \overset{\simeq}{\lra} & \Dol(E) \\
d_A & \lmt & d_A^{\, 0,1}
\end{array}$$

\noindent between the space of unitary connections on $(E,h)$ and the space of Dolbeault operators on $E$.
\end{corollary}

Recall that we denote $\cG_h=\Ga(\U(E,h))$ the group of unitary automorphisms of $(E,h)$. It is commonly called the \textit{unitary gauge group}. As for the group $\cG_E=\Ga(\GL(E))$ of all complex linear automorphisms of $E$, it is commonly called the \textit{complex gauge group}. A good reason for this terminology is that $\cG_E$ is actually the complexification of $\cG_h$ (indeed $\GL(r,\C)$ is the complexification of $\U(r)$, so the typical fibre of $\GL(E)$ is the complexification of the typical fibre of $\U(E,h)$). We saw in Section \ref{Dolbeault_op} that, over a Riemann surface $\Si$, the set of isomorphism classes of holomorphic structures on $E$ was in bijection with the orbit space $$\Dol(E) / \, \cG_E\, .$$ Now if $E$ has Hermitian metric $h$, we can replace, as we have just seen, $\Dol(E)$ with $\cA(E,h)$, and then use the bijection between the two to transport the $\cG_E$-action from $\Dol(E)$ to $\cA(E,h)$. Computing through this procedure gives, for all $u\in\cG_E$ and all $d_A\in\cA(E,h)$, the relation 

\begin{equation}\label{extension_of_unitary_gauge_action}
u\cdot d_A = d_A -\big[(d_A^{\, 0,1}u)u^{-1} - \big((d_A^{\, 0,1}u)u^{-1}\big)^*\big]
\end{equation}

\noindent where $d_A^{\, 0,1}u$ denotes the $\C$-antilinear part of the covariant derivative of the \textit{endomorphism} $u$ (this extension of a Dolbeault operator on a bundle to endomorphisms of that bundle was discussed at the end of Section \ref{Dolbeault_op}) and $\alpha^*$ denotes the $h$-unitary adjoint of an $E$-valued, or an $\End{E}$-valued, $k$-form $\alpha$ (the proof of relation \eqref{extension_of_unitary_gauge_action} is proposed as an exercise in Exercise \ref{proof_extension_of_unitary_gauge_action}). In particular, if $u$ actually lies in $\cG_h\subset \cG_E$, then $u^*=u^{-1}$, and $$(d_A^{\, 0,1}u)^* = d_A^{\, 1,0}(u^*) = d_A^{\, 1,0}(u^{-1}) = -u(d_A^{\, 1,0}u)u^{-1}$$ so

\begin{eqnarray*}
u\cdot d_A & = & d_A - (d_A^{\, 0,1}u+d_A^{\, 1,0}u)u^{-1} \\
& = & d_A - (d_A u)u^{-1}\, ,
\end{eqnarray*}

\noindent which is no other than the natural action of $\cG_h$ on $\cA(E,h)$, defined for all $s\in\Om^0(M;E)$ by $$(u\cdot d_A)(s) = u\big(d_A(u^{-1}s)\big)$$ (the exact formal analogue of the $\cG_E$-action on $\Dol(E)$, see Equations \eqref{action_on_Dolbeault_op} and \eqref{gauge_action_Dolbeault}). The fact that the action of $\cG_h$ on $\cA(E,h)$ extends to an action of $\cG_E=\cG_h^{\C}$ is what eventually explains the relation between the symplectic picture and the Geometric Invariant Theoretic picture for vector bundles on a curve, a relation which plays an important part in Donaldson's Theorem. 

To sum up, the choice of a Hermitian metric on a smooth complex vector bundle $E\lra \Si$ over a Riemann surface provides an identification between the set of isomorphism classes of holomorphic structures on $E$ and the orbit space $$\cA(E,h)/\, \cG_E\, .$$ This raises the question: what happens if we choose a different metric? If $h$ and $h'$ are two Hermitian metrics on $E$, then there exists an automorphism $u\in\cG_E$ (in fact unique up to mutiplication by an element of $\cG_h$) such that $$h'=u^*h$$ (meaning that, for any pair $(s_1,s_2)$ of smooth sections of $E$, one has $h'(s_1,s_2) = h(us_1,us_2)$). In particular, a linear connection $D$ on $E$ is $h'$-unitary if and only if the linear connection $u\cdot D= u(D(u^{-1}\,\cdot\,))$ is $h$-unitary. Indeed,

\begin{eqnarray*}
d\big(h(s_1,s_2)\big) & = & d \big(h'(u^{-1}s_1,u^{-1}s_2)\big) \\
& = & h'\big(D(u^{-1}s_1),u^{-1}s_2\big) + h'\big(u^{-1}s_1,D(u^{-1}s_2)\big) \\
& = & h \big( (u\cdot D)(s_1),s_2\big) + h\big(s_1,(u\cdot D)(s_2)\big)\, .
\end{eqnarray*}

\noindent Therefore, there is a non-canonical bijection $\cA(E,h') \simeq \cA(E,h)$ \textit{with the key property that it sends $\cG_E$-orbits to $\cG_E$-orbits}. In particular, there is a \textit{canonical} bijection $$\cA(E,h') /\, \cG_E \simeq \cA(E,h) /\, \cG_E\, .$$ This renders the choice of the metric unimportant in the whole analysis of holomorphic structures on $E$: the space $\cA(E,h)$ depends on that choice, but not the space $\cA(E,h)/\, \cG_E$, which is the space of isomorphism classes of holomorphic structures on $E$.

\subsection{The Atiyah-Bott symplectic form}

Only from this point on does it become truly necessary to assume that the base manifold of our holomorphic bundles be a \textit{compact}, \textit{connected} Riemann surface $\Si_g$ ($g$ being the genus). The fact that $\Si_g$ is of complex dimension one has already been used, though, for instance to show that \textit{any} unitary connection on a smooth Hermitian vector bundle $(E,h)$ over $\Si_g$ defines a holomorphic structure on $E$. We shall now use the compactness of $\Si_g$ to show that $\cA(E,h)$ has a natural structure of infinite-dimensional symplectic (in fact, K\"ahler) manifold. Actually, for this to be true, we would need to amend our presentation of Dolbeault operators and unitary connections to allow non-smooth such operators. Indeed, as the vector space $\Om^1(\Si;\frak{u}(E,h))$ on which the affine space $\cA(E,h)$ is modelled is infinite-dimensional, we have to choose a topology on it. In order to turn the resulting topological vector space into a Banach space, we have to work with connections which are not necessarily smooth, but instead lie in a certain Sobolev completion of the space of smooth connections, and the same goes for gauge transformations. We refer to \cite{AB} (Section 13) and \cite{Don_NS,DK} for a discussion of this problem. Atiyah and Bott have in particular shown that gauge orbits of such unitary connections always contain smooth connections, and that two smooth connections lying in a same gauge orbit can also be conjugated by a smooth gauge transformation. These analytic results enable us to ignore the issue of having to specify the correct connection spaces and gauge groups, and focus on the geometric side of the ideas of Atiyah-Bott and Donaldson instead.

Recall that the space $\cA(E,h)$ of unitary connections on a smooth Hermitian vector bundle $(E,h)$ is an affine space whose group of translations is the space $\Om^1(\Si_g;\frak{u}(E,h))$ of $1$-forms with values in the bundle of anti-Hermitian endomorphisms of $(E,h)$. In particular, the tangent space at $A$ to $\cA(E,h)$ is canonically identified with $\Om^1(\Si_g;\frak{u}(E,h))$. We assume throughout that the Riemann surface $\Si_g$ comes equipped with a compatible Riemannian metric of normalised unit volume. Compatibility in the present context means that the complex structure $I$ on each tangent plane to $\Si_g$ is an isometry of the Riemannian metric. This defines in particular a symplectic form, also a volume form since $\dim_{\R}\Si_g=2$, namely $\vol_{\Si_g}=g(I\cdot\, |\, \cdot)$. The typical fibre of $\frak{u}(E,h)$ is the Lie algebra $\frak{u}(r)$ of anti-Hermitian matrices of size $r$, so it has a canonical, positive definite inner product $$\kappa:= -\tr : \begin{array}{ccc}
\fu(r) \otimes \fu(r) & \lra & \R \\ (X;Y) & \lmt & - \tr(XY) \end{array}$$ (the restriction to $\fu(r)$ of the canonical Hermitian product $(X,Y) \lmt -\tr(\ov{X}^t\, Y)$ of $\gl(r,\C)$).

Given $A\in \cA(E,h)$ and $a,b\in T_A\cA(E,h) \simeq \Om^1(\Si_g;\fu(E,h))$, $a\wedge b$ is the $\fu(E,h)\otimes \fu(E,h)$-valued $2$-form defined by $$(a\wedge b)_x(v,w) = (a_x(v) \otimes b_x(w) - b_x(v) \otimes a_x(w)) \in \fu(r)\otimes\fu(r).$$ So $$\kappa(a\wedge b)_x(v,w) := -\tr(a_x(v)b_x(w)) + \tr(b_x(v) a_x(w))$$ is an $\R$-valued $2$-form on $\Si_g$. Note indeed that $\kappa(b\wedge a) = - \kappa(a\wedge b)$ because $b\wedge a = -a\wedge b$. Since $\Si_g$ is oriented and compact, the integral $$\w_A(a,b):= \int_{\Si_g} \kappa(a\wedge b)\in \R$$ defines a $2$-form on $\cA(E,h)$.

\begin{proposition}[Atiyah-Bott]\label{non_degeneracy}
The $2$-form $\w$ defined on $\cA(E,h)$ by $$\w_A(a,b):= \int_{\Si_g} \kappa(a\wedge b)$$ is a symplectic form.
\end{proposition}

\begin{proof}
$\w$ is obviously closed, since it is constant with respect to $A$. To show that it is non-degenerate, we use local coordinates. The tangent vectors $a$ and $b$ become $\fu(r)$-valued $1$-forms on an open subset $U\subset \Si_g$,
$$\begin{array}{rcl}
a & = & \alpha\, dx + \beta\, dy \\
b & = & \gamma\, dx + \delta\, dy
\end{array}
$$ with $\alpha, \beta, \gamma, \delta:U \lra \fu(r)$ smooth functions. If $a\in \ker\w_A$, then, for $b=\ast a:=-\beta\, dx + \alpha\, dy$, one has $$\kappa(a\wedge b)_{(x,y)}(v,w) = (\underbrace{\kappa(\alpha(x,y)^2 + \beta(x,y)^2)}_{\geq 0}) (v_1w_2-v_2w_1),$$ a positive multiple of the volume form (here we need the local coordinates $(x,y)$ to be appropriately chosen), so $$\int_{\Si_g} \kappa(a\wedge \ast a) \geq 0$$ and it is $0$ if and only if $\alpha=\beta=0$, i.e. $a=0$.
\end{proof}

\noindent Of course, there is some hidden meaning to this proof: the transformation

\begin{equation}\label{local_exp_of_the_Hodge_star}
\ast: \alpha\, dx + \beta\, dy \lmt -\beta\, dx + \alpha\, dy
\end{equation}

\noindent is the local expression of the \textit{Hodge star} on $\Om^1(\Si_g;\fu(E,h))$. It squares to minus the identity, so it is a complex structure on $\Om^1(\Si_g;\fu(E,h))$. But in fact, the Hodge star may be defined on all non-zero homogeneous forms on $\Si_g$: it sends $0$-forms to $2$-forms and vice versa, the two transformations being inverse to one another. Locally, one has $\ast (fdx) = fdy$, $\ast (fdy) = -fdx$, $\ast f = f dx\wedge dy$, and $\ast(f dx\wedge dy) = f$. More intrisically, since $\Si_g$ has a Riemannian metric and the fibres of $\fu(E,h)$ have a scalar product $\kappa$, the bundle $\bigwedge^k T^*\Si_g \otimes_{\R} \fu(E,h)$ has a Riemannian metric $\pi$, say. If $a,$ are two $\fu(E,h)$-valued $k$-forms on $\Si_g$, i.e. two sections of $\bigwedge^kT^*\Si_g \otimes_{\R} \fu(E,h)$, then $\pi(a,b)$ is a smooth function on $\Si_g$. Now, if $\eta$ is an arbitrary $\fu(E,h)$-valued $k$-form on $\Si_g$, $\ast \eta$ is defined as the unique $\fu(E,h)$-valued $(2-k)$-form such that 

\begin{equation}\label{intrinsic_def_of_the_Hodge_star}
\kappa(\eta\wedge\ast\eta) = \pi(\eta,\eta)\, \vol_{\Si_g}
\end{equation}

\noindent as $2$-forms on $\Si_g$.

\begin{proposition}
Set, for all $a,b\in T_A\cA(E,h) \simeq \Om^1(\Si_g;\fu(E,h))$, $$(a\, |\, b)_{L^2} := \int_{\Si_g} \kappa(a\wedge\ast b) = \w_A(a,\ast b).$$ Then $(\cdot\, |\, \cdot)_{L^2}$ is a Riemannian metric on $\cA(E,h)$, called the $L^2$ metric. The Atiyah-Bott symplectic form $\w$, the complex structure $\ast$, and the metric $(\cdot\, |\, \cdot)_{L^2}$ turn $\cA(E,h)$ into a K\"ahler manifold.
\end{proposition}

\begin{proof}
Note that $(a\, |\, b)_{L^2}= \int_{\Si_g} \pi(a,b)\, \vol_{\Si_g}$. The equality with the expression in the statement of the Proposition follows from \ref{intrinsic_def_of_the_Hodge_star}. The fact that $(\cdot\, |\, \cdot)_{L^2}$ is positive-definite has been proved in Proposition \ref{non_degeneracy}. Moreover, it is clear from either of expressions \ref{local_exp_of_the_Hodge_star} or \ref{intrinsic_def_of_the_Hodge_star}, that $\|\ast a\|_{L^2} = \|a\|_{L^2}$. The rest is the definition of a K\"ahler manifold (see for instance \cite{McDuff}).
\end{proof}

Recall now that the gauge group $\cG_h=\Gamma(\U(E,h))$ of unitary transformations of $(E,h)$ acts on $\cA(E,h)$ via $$u\cdot A = A -(d_Au)u^{-1}.$$

\begin{proposition}[Infinitesimal gauge action]\label{infinitesimal_action}
The fundamental vector field $$\xi^{\#}_A = \frac{d}{dt}|_{t=0} \left(\exp(t\xi)\cdot A\right)$$ associated to the element $\xi$ of the Lie algebra $\Om^0(\Si_g;\fu(E,h))\simeq Lie(\cG_h)$ is $$\xi^{\#}_A = -d_A\xi \in \Om^1(\Si_g;\fu(E,h))\, .$$
\end{proposition}

\begin{proof}
In local coordinates, $A$ is of the form $d+a$, where $a$ is a $\fu(r)$-valued $1$-form defined on an open subset $U\subset \Si_g$, $u$ is a smooth map $U\lra\U(r)$, and $d_Au$ acts on endomorphism of $E|_U$ by $d+[a,\cdot]$ (see Exercise \ref{local_expression_of_the_covariant_derivarive_of_an_endom}). So $u\cdot A$ is of the form

\begin{eqnarray*}
(d+a) - (du+[a,u])u^{-1} & = & d+a - (du)u^{-1} - (au -ua)u^{-1} \\
& = & d - (du)u^{-1} + uau^{-1}.
\end{eqnarray*}

\noindent Setting $u=\exp(t\xi)$ and taking the derivative at $t=0$ of $-(du)u^{-1} + uau^{-1}$, we obtain $$-d\xi +\xi a -a\xi = -d\xi - [a,\xi],$$ which is the local expression of $-d_A\xi$.
\end{proof}

\begin{proposition}
The action of $\cG_h$ on $\cA(E,h)$ preserves the Atiyah-Bott symplectic form and the $L^2$ metric on $\cA(E,h)$.
\end{proposition}

\begin{proof}
The tangent map to the action of $u\in\cG_h$ on $\cA(E,h)$ is the map

$$\begin{array}{ccc}
\Om^1(\Si_g;\fu(E,h)) & \lra & \Om^1(\Si_g;\fu(E,h)) \\
a & \lmt & uau^{-1}
\end{array}$$

\noindent so, since $\kappa=-\tr$ is $\mathrm{Ad}$-invariant on $\fu(r)\otimes\fu(r)$, $$\kappa\big((uau^{-1}) \wedge (ubu^{-1})\big) = \kappa \big(a\wedge b\big)$$ and therefore $u^*\w=\w$. Since the action is also $\C$-linear (see Exercise \ref{C_linear_action}), it is an isometry of the $L^2$ metric.
\end{proof}

\noindent Since we have a symplectic action of a Lie group $\cG_h$ on a symplectic manifold $(\cA(E,h),\w)$ (albeit both infinite-dimensional), it makes sense to ask whether this action is Hamiltonian and, more importantly, find the momentum map. 
To identify a possible momentum map, we need to make $(Lie(\cG_H))$ more explicit.

\begin{proposition}\label{identification_of_dual_of_Lie_algebra}
The map 

$$\begin{array}{ccc}
\Om^2(\Si_g;\fu(E,h)) & \lra & (Lie(\cG_h))^* \\
R & \lmt & (\xi \lmt \int_{\Si_g} \kappa(\xi\otimes R)
\end{array}$$

\noindent is an isomorphism of vector spaces which is $\cG_h$-equivariant with respect to the action $u\cdot R := \mathrm{Ad}_{u} \circ R$ on $\Om^2(\Si_g;\fu(E,h))$ and the co-adjoint action on $(Lie(\cG_h))^*$.
\end{proposition}

\begin{proof}
The Lie algebra of $\cG_h$ is $\Om^0(\Si_g;\fu(E,h))$. It carries a Riemannian metric $$(\lambda,\mu) \lmt \int_{\Si_g} \kappa(\lambda\wedge \ast\mu)$$ which canonically identifies it with its dual. Then, the Hodge star establishes an isomorphism $$\ast:\Om^2(\Si_g;\fu(E,h)) \lra \Om^0(\Si_g;\fu(E,h))\, .$$ The statement on the action follows from the $\Ad{}$-invariance of $\kappa$.
\end{proof}

\noindent Now, there is a natural map from $\cA(E,h)$ to $\Om^2(\Si_g;\fu(E,h))$, namely the map taking a unitary connection $A$ to its curvature $F_A$, which we now define.

\begin{proposition}\label{def_curvature}
A unitary connection $$d_A: \Om^0(\Si_g;E) \lra \Om^1(\Si_g;E)$$ on $(E,h)$ uniquely extends to an operator $$d_A: \Om^k(\Si_g;E) \lra \Om^{k+1} (\Si_g;E)$$ satisfying the generalised Leibniz rule $$d_A(\beta \wedge \si) = (d\beta) \wedge \si + (-1)^{\deg{\beta}} \beta\wedge d_A\si$$ for all $\beta \in \Om^j(\Si_g;\C)$ and all $\si\in\Om^k(\Si_g;E)$. The operator $$d_A\circ d_A:\Om^0(\Si_g;E) \lra \Om^2(\Si_g;E)$$ is $C^{\infty}(\Si_g;\C)$-linear, so it defines an element $F_A \in \Om^2(\Si_g;\fu(E,h))$ called the \textbf{curvature} of $A$. It satisfies $$F_{u\cdot A} = \mathrm{Ad}_{u} \circ F_A = uF_Au^{-1}$$ for all $u\in\cG_h$. Moreover, if the local expression of $A$ is $d+a$, the local expression of $F_A$ is $da+\frac{1}{2} [a,a]$.
\end{proposition}

\noindent The following theorem is the main result of this subsection.

\begin{theorem}[Atiyah-Bott, \cite{AB}]\label{curvature_equals_momentum}
The curvature map $$F: \cA(E,h)\lra \Om^2(\Si_g;\fu(E,h))$$ is an equivariant momentum map for the gauge action of $\cG_h$ on $\cA(E,h)$.
\end{theorem}

\noindent We shall need the following lemma to prove Theorem \ref{curvature_equals_momentum}.

\begin{lemma}\label{curvature_of_a_translate}
Let $A\in\cA(E,h)$ be a unitary connection and let $b\in\Om^1(\Si_g;\fu(E,h))$. Then $A+b$ is a unitary connection and $$F_{A+b} = F_A + d_A b + \frac{1}{2}[b,b].$$ 
\end{lemma}

\begin{proof}
Since $\cA(E,h)$ is an affine space on $\Om^1(\Si_g;\fu(E,h))$, $A+b$ is a unitary connection. Let $d+a$ be the local expression of $A$, where $a\in\Om^1(U;\fu(r))$. Then the local expression of $F_A$ is $da+\frac{1}{2}[a,a]$, and the local expression of $d_A b$ is $db+[a,b]$ (as in Exercise \ref{local_expression_of_the_covariant_derivarive_of_an_endom}). Moreover, the local expression of $A+b$ is $d+(a+b)$, so the local expression of $F_{A+b}$ is
\begin{eqnarray*}
d(a+b) + \frac{1}{2} [a+b,a+b] & = & da + db + \frac{1}{2}[a,a] + [a,b] + \frac{1}{2}[b,b] \\
& = & (da + \frac{1}{2}[a,a]) + (db + [a,b]) + \frac{1}{2}[b,b]
\end{eqnarray*}
\noindent so indeed $$F_{A+b}= F_A + d_A b +\frac{1}{2}[b,b]\, .$$
\end{proof}

\begin{proof}[Proof of Theorem \ref{curvature_equals_momentum}]
The equivariance of $F$ follows from Proposition \ref{def_curvature}. It remains to show that $F$ is a momentum map for the gauge action, that is, for all $\xi\in Lie(\cG_h)=\Om^0(\Si_g;\fu(E,h))$ and all $A\in\cA(E,h)$, $$\w_A(\xi^{\#}_A,\,\cdot\,) = \big(d<F\,,\,\xi>)_A\big(\,\cdot\,\big)$$ as linear forms on $T_A\cA(E,h)\simeq\Om^1(\Si_g;\fu(E,h))$. By Proposition \ref{infinitesimal_action}, this is equivalent to the fact that, for all $\eta\in\Om^1(\Si_g:\fu(E,h))$, $$\int_{\Si_g}\kappa(-d_A\xi \wedge \eta) = < (dF)_A\cdot \eta\,,\,\xi>.$$ But, by Proposition \ref{curvature_of_a_translate}, $$F_{A+t\eta} = F_A + t d_A\eta + \frac{1}{2} t^2 [\eta,\eta]\, ,$$ so $$(dF)_A\cdot \eta= \frac{d}{dt}|_{t=0} F_{A+t\eta} = d_A\eta\, .$$ In other words, by Proposition \ref{identification_of_dual_of_Lie_algebra}, we want to show that 

\begin{equation}\label{momentum_relation}
-\int_{\Si_g} \kappa(d_A\xi\wedge\eta) = \int_{\Si_g} \kappa(\xi\otimes d_A\eta)\,.
\end{equation} 

\noindent But, since $\partial\Si_g=\emptyset$, one has $$\int_{\Si_g} d\big(\kappa(\xi\otimes\eta)\big)=0$$ on the one hand, and on the other hand, $$d\big(\kappa(\xi\otimes\eta)\big) = \kappa(d_A\xi \wedge \eta) + \kappa(\xi\otimes d_A\eta)\,$$ whence relation \eqref{momentum_relation}.
\end{proof}

\newpage

\subsection{Exercises}\hfill

\begin{multicols}{2}

\begin{exercise}
Show that any complex vector bundle over a smooth manifold admits a Hermitian metric (as usual, use local trivialisations and a partition of unity).
\end{exercise}

\begin{exercise}
Let $u$ be an endomorphism of a smooth Hermitian vector bundle $(E,h)$. Show that there exists a unique endomorphism $u^*$ of $E$ such that, for all $(s,s')\in \Ga(E)\times\Ga(E)$, $$h\big(u(s),s'\big) = h\big(s,u^*(s')\big).$$ $u^*$ is called the \textit{adjoint} of $u$. A Hermitian endomorphism is self-adjoint, and an anti-Hermitian one is anti-self-adjoint.
\end{exercise}

\begin{exercise}\label{proof_extension_of_unitary_gauge_action}
Show that, if $d_A\in\cA(E,h)$ and $g\in\cG_E$, then the quantity $g\cdot d_A$ defined by $$d_A - \big[(d_A^{\, 0,1}g)g^{-1} - \big((d_A^{\, 0,1}g)g^{-1}\big)^*\big]$$ is a unitary connection, and that this defines an action of $\cG_E$ on $\cA(E,h)$ making the isomorphism $$\cA(E,h) \simeq \Dol(E)$$ $\cG_E$-equivariant.
\end{exercise}

\begin{exercise}
Check that relation \eqref{intrinsic_def_of_the_Hodge_star} gives a well-defined $\R$-linear map $$\ast: \Om^k(\Si_g;\fu(E,h)) \lmt \Om^2(\Si_g;\fu(E,h))$$ satisfying $\ast^2 = (-1)^{k(2-k)} \Id$. Check that, in local coordinates, the map $\ast$ satisfies $$\ast(\alpha\, dx +\beta\, dy) = -\beta\,dx + \alpha\, dy\,.$$ How about $\ast(\lambda\, dz + \mu\, d\ov{z})$ ?
\end{exercise}

\begin{exercise}\label{C_linear_action}
Show that the tangent map to the self-diffeomorphism of $\cA(E,h)$ defined by the action of an element $u\in\cG_h$ is $\C$-linear with respect to the complex structure of $\cA(E,h)$ given on each tangent space $T_A\cA(E,h) \simeq \Om^1(\Si_g;\fu(E,h))$ by the Hodge star.
\end{exercise}

\begin{exercise}\label{local_expression_of_the_covariant_derivarive_of_an_endom}
Let $A$ be a linear connection on a vector bundle $E$, and let $s$ be a section of $E$. Show that if $A$ is locally of the form $$s\lmt ds + as$$ then the covariant derivative $d_Au$ of an endomorphism of $E$, defined at the end of Section \ref{Dolbeault_op} by $$(d_Au)s = d_A(u(s)) - u(d_As)\, ,$$ is locally of the form $$u\lmt du+[a,u].$$
\end{exercise}
\end{multicols}

\section{Moduli spaces of semi-stable vector bundles}\label{stable_bundles}

It is sometimes important, while thinking about mathematics, to have a guiding problem to help one organise one's thoughts. For us in these notes, it is the problem of classifying holomorphic vector bundles on a smooth, irreducible complex projective curve $\Si_g$ (=a compact connected Riemann surface of genus $g$). When the genus is $0$ or $1$, there are complete classification results for holomorphic vector bundles on $\Si_g$, due to Grothendieck for the case of the Riemann sphere (\cite{Grot_P1}), and to Atiyah for the case of elliptic curves (\cite{Atiyah_elliptic_curves}). There are no such classification results available for holomorphic vector bundles on a curve of genus $g> 1$. In such a situation, one generally hopes to replace the classification theorem by the construction of what is called a \textit{moduli space}, the geometry of which can subsequently be studied. Roughly speaking, a moduli space of holomorphic vector bundles is a complex quasi-projective variety which has isomorphism classes of vector bundles over a fixed base for points, and satisfies a universal property controlling the notion of holomorphic or algebraic family of such vector bundles. We shall not get into the formal aspects of the notion of a moduli space and we refer the interested reader to \cite{Gomez} instead. There are a few situations in which we know how to construct a moduli variety of vector bundles (that is, give a structure of complex quasi-projective variety to \textit{a certain set} of equivalence classes of vector bundles) and vector bundles on a smooth complex projective curve is one of those situations. The difficulty of a \textit{moduli problem} is to understand \textit{which set} one should try to endow with a structure of complex quasi-projective variety.

Common features of many moduli problems include:
\begin{enumerate}
\item Starting with a topological (or smooth) classification of the objects under study. This is typically obtained via \textit{discrete invariants} (for vector bundles on curves: the rank and the degree) and has the virtue of dividing the moduli problem into various, more tractable moduli problems for objects of a fixed topological type.
\item Getting rid of certain objects in order to get a moduli space that admits a structure of projective algebraic variety (=a closed subspace of a projective space), or at least quasi-projective (=an open subset of a projective variety). This is where continuous invariants, called moduli, enter the picture (moduli may be thought of as some sort of local coordinates on the would-be moduli space). It is usually a difficult problem to find moduli for a class of objects, and one solution has been to use Mumford's Geometric Invariant Theory (GIT, \cite{Mumford_GIT}) to decide which objects one should consider in order to get a nice moduli space (these objects are called \textit{semi-stable} objects).
\end{enumerate}

\noindent In fact, \textit{stable objects} exhibit even better properties in the sense that the moduli space is then typically an orbit space (also called a geometric quotient, as opposed to a categorical quotient in the semi-stable case, see for instance \cite{Thomas_GIT,Newstead_GIT}) admitting a structure of  quasi-projective variety. GIT really is a way of defining quotients in algebraic geometry, and it has been applied very successfully to the study of moduli problems (Mumford's original motivation indeed). We shall not say anything else about GIT in these notes, and focus on \textit{slope stability} for vector bundles on a curve only (it can be shown that this is in fact a GIT type of stability condition, see for instance \cite{Newstead_GIT}). Nor shall we say anything about moduli functors and their coarse/fine moduli spaces (the interested reader might consult, for instance, \cite{Mukai}).

\subsection{Moduli spaces of line bundles}

Let $\Si_g$ be a compact, connected Riemann surface of genus $g$. As a first step into the moduli problem for holomorphic vector bundles on $\Si_g$, let us recall that isomorphism classes of holomorphic \textit{line} bundles of degree $d$ on $\Si_g$ can be arranged into a smooth, projective variety called the Picard variety of degree $d$, and denoted $\Pic^d(\Si_g)$. A nice construction of $\Pic^d(\Si_g)$ can be given using sheaf cohomology. Some of the techniques we use below are not differential-geometric and this subsection can be skipped without harm if need be.

Denote $\O$ the sheaf of holomorphic functions on $\Si_g$ (local sections of $\O$ are holomorphic functions $f:U\lra\C$ defined on an open subset $U$ of $\Si_g$), and $\O^*$ the sheaf of nowhere vanishing holomorphic functions $(f:U\lra \C^*)$. Since $\C^*=\GL(1,\C)$, the first \v{C}ech cohomology group $\ceH(\Si_g,\O^*)$ is the set of isomorphism classes of holomorphic line bundles on $\Si_g$ (this set is a group because $\GL(1,\C)=\C^*$ is an Abelian group). But $\O^*$ fits into the following short exact sequence of sheaves $$0\lra \underline{\Z} \lra \O \underset{\exp}{\lra} \O^* \lra 1$$ (for any open subset $U\subset \Si_g$, the kernel of $\exp_U:f\lmt \exp(i2\pi f)$ consists of $\Z$-valued locally constant functions on $U$, and a nowhere vanishing holomorphic function $g:U\lra \C^*$ is \textit{locally} the exponential of a holomorphic function). There is a long exact sequence in cohomology associated to this short exact sequence (see for instance \cite{Griffiths-Harris}), which starts as follows.

$$\xymatrix{
0 \ar[r] & \ceH^0(\Si_g;\underline{\Z}) \ar[r] & \ceH^0(\Si_g;\O) \ar[r] & \ceH^0(\Si_g;\O^*)  \ar@{->} `r/8pt[d] `d/10pt[l] `^dl[ll] `^r/3pt[dll] [dll] \\
& \ceH^1(\Si_g;\underline{\Z}) \ar[r] & \ceH^1(\Si_g;\O) \ar[r] & \ceH^1(\Si_g;\O^*)  \ar@{->}`r/8pt[dll]`d/10pt[l] `^dl[ll] `^r/3pt[dll] [dll] \\
& \ceH^2(\Si_g;\underline{\Z}) \ar[r] & \ceH^2(\Si_g;\O)= 0
}
$$

\noindent The proof uses Dolbeault's Theorem (see for instance \cite{Griffiths-Harris}), which establishes the comparison theorem $\ceH^2(\Si;\mathcal{O}_{\Si}) \simeq H^{\, 0,2}_{\mathrm{Dol}}(\Si)$, between the \v{C}ech cohomology of $\Si$ with coefficients in the structure sheaf $\mathcal{O}_{\Si}$ and the cohomology of the complex of $(0,q)$-forms on $\Si$: a Riemann surface has complex dimension one, so all $(0,2)$-forms are zero (locally, one has $d\ov{z}\wedge d\ov{z}=0$) and therefore $H^{\, 0,2}_{\mathrm{Dol}}(\Si)=0$. There is also a comparison theorem for the \v{C}ech cohomology groups $\ceH^{k}(\Si_g;\underline{\Z})$, where the coefficient sheaf is a locally constant sheaf. It gives an isomorphism between these groups and the singular cohomology groups $H^{k}_{\mathrm{sing}}(\Si;\Z)$ (see for instance \cite{Bott-Tu}). One may combine this with the above to show that all the remaining terms in the cohomology long exact sequence are in fact zero. As for us, we use the comparison with topology to obtain that, since $\Si_g$ is connected, compact and orientable, $$\ceH^0(\Si_g;\underline{\Z}) = \Z,\ \ceH^1(\Si_g;\underline{\Z})=\Z^{2g},\ \mathrm{and}\ \ceH^2(\Si_g;\underline{\Z})=\Z.$$ So the first line in the exact sequence above simply is $$0\lra \Z \lra \C \underset{\exp}{\lra} \C^*.$$ Since the exponential is surjective and the whole sequence is exact, the rest of it writes $$0\lra \ceH^1(\Si_g;\underline{\Z}) \lra \ceH^1(\Si_g;\O) \lra \ceH^1(\Si_g;\O^*) \lra \ceH^2(\Si_g;\underline{\Z})\lra 0\, .$$ So we obtain the short exact sequence 

\begin{equation}\label{Jacobian_and_Picard}
0\lra \frac{\ceH^1(\Si_g;\O)}{\ceH^1(\Si_g;\underline{\Z})} \lra \ceH^1(\Si_g;\O^*) \lra \ceH^2(\Si_g;\underline{\Z}) \lra 0\, .
\end{equation}

\noindent Since moreover $\ceH^1(\Si_g;\O)\simeq \C^g$ (see for instance \cite{Griffiths-Harris}), we have in fact the following short exact sequence of group homomorphisms $$0\lra \C^g/\Z^{2g} \lra \ceH^1(\Si_g;\O^*) \lra \Z\lra 0\, .$$ 

\begin{definition}\label{def:Jacobian_and_Picard}
The Abelian group $$\frac{\ceH^1(\Si_g;\O)}{\ceH^1(\Si_g;\underline{\Z})}$$ is denoted $\Jac(\Si_g)$ and called the \textbf{Jacobian} of $\Si_g$. It is isomorphic as a group to $\C^g/\Z^{2g}$, and therefore has a natural structure of compact complex analytic manifold of complex dimension $g$. The group multiplication and inverse are analytic maps with respect to this structure, so $\Jac(\Si_g)$ is a compact, Abelian complex Lie group.

The Abelian group $\ceH^1(\Si_g;\O^*)$ is called the \textbf{Picard group} of $\Si_g$, denoted $\Pic(\Si_g)$. The group structure on it is induced by the tensor product of line bundles, the inverse of a line bundle being its dual.
\end{definition}

\noindent Finally, call \textit{degree} the surjective map $$\mathit{deg}:\Pic(\Si_g)\lra \Z$$ in the short exact sequence above. We are going to show that two holomorphic line bundles $\cL$ and $\cL'$ are topologically isomorphic if and only if $\deg\cL=\deg\cL'$. We denote $\cC_{\Si_g}$ the sheaf of continuous complex-valued functions on $\Si_g$, and $\cC_{\Si_g}^{\, *}$ the sheaf of nowhere vanishing such functions. As earlier, there is an exact sequence $$0\lra\underline{\Z} \lra \cC_{\Si_g} \underset{\exp}{\lra} \cC_{\Si_g}^{\, *} \lra 1\,.$$ The associated long exact sequence in cohomology starts as follows.

$$\xymatrix{
0 \ar[r] & \ceH^0(\Si_g;\underline{\Z}) \ar[r] & \ceH^0(\Si_g;\cC_{\Si_g}) \ar[r] & \ceH^0(\Si_g;\cC_{\Si_g}^{\, *})  \ar@{->} `r/8pt[d] `d/10pt[l] `^dl[ll] `^r/3pt[dll] [dll] \\
& \ceH^1(\Si_g;\underline{\Z}) \ar[r] & \ceH^1(\Si_g;\cC_{\Si_g}) \ar[r] & \ceH^1(\Si_g;\cC_{\Si_g}^{\, *})  \ar@{->}`r/8pt[dll]`d/10pt[l] `^dl[ll] `^r/3pt[dll] [dll] \\
& \ceH^2(\Si_g;\underline{\Z}) \ar[r] & \ceH^2(\Si_g;\cC_{\Si_g})\, .
}
$$

\noindent One notable difference with the holomorphic case is the fac that there exists non-constant continuous functions on $\Si_g$, so the exact sequence above does not start in the same way as in the holomorphic case (indeed, the map $\ceH^0(\Si_g;\cC_{\Si_g}) \lra \ceH^0(\Si_g;\cC_{\Si_g}^{\, *})$ is not surjective). But let us focus on the part of the exact sequence that most interests us, namely

\begin{equation}\label{deg_isom}
\ceH^1(\Si_g;\cC_{\Si_g}) \lra \ceH^1(\Si_g;\cC_{\Si_g}^{\, *})  \lra \ceH^2(\Si_g;\underline{\Z}) \lra \ceH^2(\Si_g;\cC_{\Si_g})\, .
\end{equation} 

\noindent Because of the existence of partitions of unity made up of continuous functions, the sheaf $\cC_{\Si_g}$ is what is called a \textit{fine} sheaf, and this implies (see for instance \cite{Griffiths-Harris}) that $$\ceH^k(\Si_g;\cC_{\Si_g}) = 0\ \mathrm{for\ all}\ k>0.$$ So, by the exactness of \eqref{deg_isom}, $$\ceH^1(\Si_g;\cC_{\Si_g}^{\, *}) \simeq \ceH^2(\Si_g;\underline{\Z})\simeq \Z\, ,$$ and we recall that $\ceH^1(\Si_g;\cC_{\Si_g}^{\, *})$ is the set of isomorphism classes of topological line bundles on $\Si_g$.

\begin{lemma}\label{top_classif_of_line_bundles}
Two holomorphic line bundles $\cL$ and $\cL'$ on a compact, connected Riemann surface $\Si_g$ are topologically isomorphic if and only if they have the same degree.
\end{lemma}

Recall that, in this subsection, the degree of a line bundle is the image of its isomorphism class under the degree map defined as above. But, the result of the lemma makes this map the degree in the sense of definition \ref{def_of_degree}.

\begin{proof}[Proof of lemma \ref{top_classif_of_line_bundles}]
Since there is a natural injective homomorphism of sheaves $\O\subset \cC_{\Si_g}$, we have a homomorphism of short exact sequences

$$\begin{CD}
0 @>>> \underline{\Z} @>>> \cC_{\Si_g} @>>\exp> \cC_{\Si_g}^{\, *} @>>> 1 \\
& & @|  @AAA @AAA \\
0 @>>> \underline{\Z} @>>> \O @>>\exp> \O^* @>>> 1
\end{CD}$$

\noindent whence we obtain the following commutative diagramme of group homomorphisms between the long exact sequences in cohomology:

$$\begin{CD}\label{proof_top_classif_of_line_bundles}
& & 0 @>>> \ceH^1(\Si_g;\cC_{\Si_g}^{\, *}) @>\simeq>> \ceH^2(\Si_g;\underline{\Z}) @>>> 0 \\
& & @AAA @AAA @| \\
0 @>>> \frac{\ceH^1(\Si_g;\O)}{\ceH^1(\Si_g;\underline{\Z})} @>>> \ceH^1(\Si_g;\O^*) @>>\mathit{deg}> \ceH^2(\Si_g;\underline{\Z}) @>>> 0
\end{CD}$$

\noindent In particular, the right square alone shows that two holomorphic line bundles $\cL$ and $\cL'$ are topologically isomorphic if and only if they have the same degree $d\in \ceH^2(\Si_g;\underline{\Z})\simeq \Z$, for then $\cL$ and $\cL'$ determine the same element in $\ceH^1(\Si_g;\cC_{\Si_g}^{\, *})$.
\end{proof}

The exact same proof, replacing $\cC_{\Si_g}$ with the sheaf of \textit{smooth} complex-valued functions on $\Si_g$ (also a fine sheaf), would show the smooth classification to be the same as the topological one. Note that the whole commutative diagramme above also shows that any topological line bundle on $\Si_g$ admits a holomorphic structure (i.e. the vertical map $\ceH^1(\Si_g;\O^*) \lra \ceH^1(\Si_g;\cC_{\Si_g}^{\, *})$ is surjective). The bottom line of the diagramme also shows that there exist holomorphic line bundles of any degree $d\in\Z$.

\begin{theorem}[Moduli spaces of holomorphic line bundles]
Let $\Si_g$ be a compact, connected Riemann surface of genus $g$. Then:
\begin{enumerate}
\item Two holomorphic line bundles on $\Si_g$ are topologically isomorphic if and only if they have the same degree $d\in\Z$.
\item For any such $d$, the set $\Pic^d(\Si_g)$ of isomorphism classes of holomorphic line bundles of degree $d$ admits a structure of compact complex analytic manifold which makes it isomorphic to the $g$-dimensional complex torus $\C^g/\Z^{2g}$ with its canonical complex analytic structure.
\end{enumerate}
\end{theorem}

\begin{proof}
The proof of (1) is Lemma \ref{top_classif_of_line_bundles}. To prove (2), we see from \eqref{Jacobian_and_Picard} and Definition \ref{def:Jacobian_and_Picard} that the set $\Pic^d(\Si_g)$ of holomorphic line bundles of degree $d\in\Z$ is the fibre of the degree map above $d$ (in particular, $\Jac(\Si_g) = \Pic^0(\Si_g)$). By the short exact sequence \eqref{Jacobian_and_Picard}, $\Pic^d(\Si_g)$ is a principal homogeneous space for the group $\ker\deg=\Jac(\Si_g)$, meaning that the action of $\Jac(\Si_g)$ on $\Pic^d(\Si_g)$ defined for $\cL_0\in\Jac(\Si_g)$ by $$\cL \lmt \cL\otimes \cL_0$$ is free and transitive. In particular, \textit{$\Pic^d(\Si_g)$ is a complex analytic manifold isomorphic to the $g$-dimensional complex torus $\C^g/\Z^{2g}$}.
\end{proof}

This shows that, while the topological classification of line bundles is achieved by means of a single discrete invariant (the degree), the holomorphic classification requires continuous parameters (the moduli of holomorphic line bundles are the points of a complex analytic manifold isomorphic to $\C^g/\Z^{2g}$). Note that here we did not have to get rid of any line bundle in order to obtain a moduli space admitting a complex analytic structure. This is no longer true in rank greater than $1$, and this is where the notion of semi-stable bundle comes into play (non-semi-stable bundles will be discarded and a moduli variety obtained) .

\subsection{Stable and semi-stable vector bundles}

A basic property of holomorphic line bundles on a compact connected Riemann surface $\Si_g$ states that they do not admit non-zero global holomorphic sections if their degree is negative (see for instance \cite{Forster}, Theorem 16.5). Since a homomorphism between the holomorphic line bundles $\cL_1$ and $\cL_2$ is a section of the \textit{line bundle} $\cL_1^*\otimes \cL_2$, a non-zero such homomorphism may only exist if $\deg(\cL_1^*\otimes \cL_2) \geq 0$, which is equivalent to $\deg{\cL_1} \leq \deg{\cL_2}$. Semi-stable vector bundles of rank $r\geq 2$ provide a class of higher rank vector bundles for which the statement above remains true (see Proposition \ref{generalising_line_bundles}). Note that, for higher rank vector bundles, the degree of $\cE_1^*\otimes \cE_2$ is $$\deg(\cE_1^*\otimes\cE_2) = \rk(\cE_1)\deg(\cE_2) - \deg(\cE_1)\rk{\cE_2}$$ so the non-negativity condition is equivalent to $$\frac{\deg{\cE_1}}{\rk{\cE_1}} \leq \frac{\deg{\cE_2}}{\rk{\cE_2}}\cdot$$ This motivates the following definition.

\begin{definition}[Slope]
The slope of a non-zero complex vector bundle $E\lra\Si_g$ on an orientable, compact, connected  surface $\Si_g$ is the rational number $$\mu(E):=\frac{\deg{E}}{\rk{E}}\in\Q\, .$$
\end{definition}

We point out that no use is made of the holomorphic structures in the definition of the slope. It is a purely topological quantity, that will, nonetheless, have strong holomorphic properties (another example of a topological invariant with strong holomorphic properties is the genus: on a compact, connected, orientable surface of genus $g$, the dimension of the space of holomorphic $1$-forms is equal to $g$ for \textit{any} complex analytic structure on the surface).

In what follows, we call a sub-bundle $\cF\subset\cE$ non-trivial if it is distinct from $0$ and $\cE$. We emphasise that the definition that we give here is that of \textit{slope} stability. However, since this is the only notion of stability that we shall consider in these notes, we will only say stable and semi-stable afterwards.

\begin{definition}[Slope stability]
A (non-zero) holomorphic vector bundle $\cE\lra\Si_g$ on a compact, connected Riemann surface $\Si_g$ is called
\begin{enumerate}
\item \textbf{slope stable}, or simply stable, if for any non-trivial holomorphic sub-bundle $\cF$, one has $$\mu(\cF) < \mu(\cE)\, .$$
\item \textbf{slope semi-stable}, or simply semi-stable, if for any non-trivial holomorphic sub-bundle $\cF$, one has $$\mu(\cF) \leq \mu(\cE)\, .$$
\end{enumerate}
\end{definition}

A couple of remarks are in order. First, all holomorphic line bundles are stable (since they do not even have non-trivial sub-bundles), and all stable bundles are semi-stable. Second, a semi-stable vector bundle with coprime rank and degree is actually stable (this only uses the definition of slope stability and the properties of Euclidean division in $\Z$). Next, we have the following equivalent characterisation of stability and semi-stability, which is sometimes useful in practice.

\begin{proposition}\label{another_charac_of_stability}
A holomorphic vector bundle $\cE$ on $\Si_g$ is stable if and only if, for any non-trivial sub-bundle $\cF\subset\cE$, one has $\mu(\cE/\cF) > \mu(\cE)$. It is semi-stable if and only if $\mu(\cE/\cF)\geq \mu(\cF)$ for all such $\cF$.
\end{proposition}

\begin{proof}
Denote $r,r',r''$ the respective ranks of $\cE,\cF$ and $\cE/\cF$, and $d,d',d''$ their respective degrees. One has an exact sequence $$0\lra \cF \lra \cE \lra \cE/\cF\lra 0$$ so $r=r'+r''$ and $d=d'+d''$. Therefore, $$\frac{d'}{r'} < \frac{d'+d''}{r'+r''} \Leftrightarrow \frac{d'}{r'} < \frac{d''}{r''} \Leftrightarrow \frac{d'+d''}{r'+r''} < \frac{d''}{r''}$$ and likewise with large inequalities or with equalities. This readily implies the Proposition.
\end{proof}

In a way, semi-stable holomorphic vector bundles are holomorphic vector bundles that do not admit \textit{too many} sub-bundles, since any sub-bundle they may have is of slope no greater than their own. This turns out to have a number of interesting consequences that we now study. We recall that the category of vector bundles on a curve is a typical example of an additive category which is not Abelian: even though it admits kernels and co-kernels (hence also images and co-images), the canonical map $\cE / \ker u \lra \Im u$ is in general not an isomorphism. We can, however, always compare the slopes of these two bundles.

\begin{lemma}\label{slopes_of_image_and_coimage}
If $u:\cE\lra \cE'$ is a non-zero homomorphism of vector bundles over $\Si_g$, then $$\mu(\cE/\ker u) \leq \mu(\Im u)$$ with equality if and only if the canonical map $\cE/ \ker u \lra \Im u$ is an isomorphism.
\end{lemma}

\noindent One says that $u$ is \textit{strict} if the canonical homomorphism $\cE/\ker u \lra \Im u$ is an isomorphism. In this case, $u$ is injective if and only if $\ker u = 0$ and $u$ is surjective if and only $\Im u=\cE'$. The proof we give below, of Lemma \ref{slopes_of_image_and_coimage}, requires notions on coherent modules over the sheaf $O_{\Si_g}$; it may be skipped upon firt reading of these notes. Recall that the category of vector bundles on $\Si_g$ is equivalent to the category of locally free $O_{\Si_g}$-modules (=torsion-free coherent $O_{\Si_g}$-modules). Let $\cE$ be a vector bundle on $\Si_g$ and let $\underline{\cE}$ be the corresponding torsion-free coherent module. Even though a coherent sub-module $\underline{\cF}$ of $\underline{\cE}$ is torsion-free, it only corresponds to a sub-bundle $\cF$ of $\cE$ if the coherent module $\underline{\cE}/\underline{\cF}$ is also torsion-free (and the latter then corresponds to the vector bundle $\cE/\cF$). This is equivalent to saying that $\underline{\cF}$ is locally a direct summand of $\underline{\cE}$. Given a coherent sub-module $\underline{\cF}$ of $\underline{\cE}$, there exists a smallest coherent sub-module $\underline{\widetilde{\cF}}$ containing $\underline{\cF}$ and such that $\underline{\cE}/\underline{\widetilde{\cF}}$ is torsion-free, namely the pre-image of the torsion sub-module of $\underline{\cE}/\underline{\cF}$. Then $\underline{\widetilde{\cF}}/\underline{\cF}$ has finite support and $\underline{\cF}$ and $\underline{\widetilde{\cF}}$ have same rank. Moreover, $\deg\underline{\cF}$ (which is well-defined since $\cF$ is locally free) satisfies $\deg\underline{\cF} \leq \deg \underline{\widetilde{\cF}}$, with equality if and only if $\underline{\widetilde{\cF}}= \underline{\cF}$. It is convenient to call the \textit{sub-bundle} $\widetilde{\cF}$ corresponding to $\underline{\widetilde{\cF}}$ the sub-bundle of $\cE$ \textit{generated} by $\underline{\cF} \subset \underline{\cE}$.

\begin{proof}[Proof of Lemma \ref{slopes_of_image_and_coimage}]
Recall that the category of coherent $O_{\Si_g}$-modules is Abelian. In particular, the isomorphism theorem between co-images and images holds in that category. Let $\underline{u}:\underline{\cE} \lra \underline{\cE'}$ be the homomorphism of coherent $O_{\Si_g}$-modules associated to $u:\cE\lra \cE'$. Then, on the one hand, the locally free $O_{\Si_g}$-module associated to $\ker u$ is $\ker \underline{u}$ and the locally free $O_{\Si_g}$-module associated to $\cE/\ker u$ is $\underline{\cE}/\ker\underline{u}$. On the other hand, $\Im u$ is the vector bundle \textit{generated} by $\Im{\underline{u}}\simeq \underline{\cE}/\underline{\ker{u}}$. So $\rk(\Im u) = \rk(\Im\underline{u})$ and $\deg(\Im{u}) \geq \deg(\Im\underline{u})$, with equality if and only if $u$ is strict. Therefore $\mu(\Im{u}) \geq \mu(\Im\underline{u})$, with equality if and only if $u$ is strict. So $$ \mu(\cE/\ker{u}) = \mu(\underline{\cE}/\ker{\underline{u}}) = \mu(\Im{\underline{u}}) \leq \mu(\Im{u})$$ with equality if and only if $\cE/\ker u \simeq \Im u$.\end{proof}

\noindent This immediately implies the result alluded to in the introduction to the present subsection.

\begin{proposition}\label{generalising_line_bundles}
Let $\cE$ and $\cE'$ be two semi-stable vector bundles such that $\mu(\cE) > \mu(\cE')$. Then any homomorphism $u:\cE \lra \cE'$ is zero.
\end{proposition}

\begin{proof}
If $u$ is non-zero, then, since $\cE$ is semi-stable, Proposition \ref{another_charac_of_stability} and Lemma \ref{slopes_of_image_and_coimage} imply that $$\mu(\Im u) \geq \mu(\cE/\ker u) \geq \mu(\cE) > \mu(\cE')\, ,$$ which contradicts the semi-stability of $\cE'$.
\end{proof}

We now focus on the category of semi-stable vector bundles of fixed slope $\mu\in\Q$. Unlike the category of all vector bundles on $\Si_g$, this is an Abelian category: it is additive, and we prove below that it admits kernels and co-kernels and that the isomorphism theorem holds.

\begin{proposition}\label{Abelian_category}
Let $u:\cE\lra\cE'$ be a non-zero homomorphism of semi-stable vector bundles of slope $\mu$. Then $\ker u$ and $\Im u$ are semi-stable vector bundles of slope $\mu$, and the natural map $\cE/\ker u\lra\Im u$ is an isomorphism. In particular, the category of semi-stable vector bundles of slope $\mu$ is Abelian.
\end{proposition}

\begin{proof}
Since $u\neq 0$, $\Im u$ is a non-zero sub-bundle of $\cE'$, so $\mu(\Im u)\leq \mu(\cE')=\mu$. But, by Lemma \ref{slopes_of_image_and_coimage}, $$\mu(\Im u) \geq \mu(\cE/\ker u) \geq \mu(\cE) = \mu\, .$$ So $\mu(\Im u) = \mu$ and $\mu(\cE/\ker u) = \mu$. In particular, by Lemma \ref{slopes_of_image_and_coimage}, $\cE/\ker u\simeq\Im u$. Consider now the exact sequence $$0\lra \ker u \lra \cE \lra \cE/\ker u \lra 0\, .$$ Since $\mu(\cE)=\mu(\cE/\ker u)=\mu$, one also has $\mu(\ker u)=\mu$. Finally, since a sub-bundle of $\ker u$ (resp. $\Im u$) is also a sub-bundle of $\cE$ (resp. $\cE'$), its slope is no greater than $\mu(\cE)=\mu=\mu(\ker u)$ (resp. $\mu(\cE')=\mu=\mu(\Im u)$), so $\ker u$ (resp. $\Im u$) is semi-stable.
\end{proof}

\noindent As an easy consequence of the above, the following result shows that, by considering only stable bundles of the same slope, we can better control the homomorphisms between them.

\begin{proposition}\label{isom_or_zero}
Let $\cE$ and $\cE'$ be two stable vector bundles on $\Si_g$ such that $\mu(\cE)=\mu(\cE')$, and let $u:\cE \lra \cE'$ be a non-zero homomorphism. Then $u$ is an isomorphism.
\end{proposition}

\begin{proof}
Recall that $\ker u\neq \cE$ by assumption. Since $u:\cE\lra\cE'$ is a non-zero homomorphism between semi-stable bundles of the same slope, Proposition \ref{Abelian_category} implies that $u$ is strict and that $\ker u$ is either $0$ or has slope equal to $\mu(\cE)$. Since $\cE$ is actually stable, $\ker u$ must be $0$. Since $u$ is strict, this implies that $u$ is injective. Likewise, $\Im u\neq 0$ by assumption, and has slope equal to $\mu(\cE')$ by Proposition \ref{Abelian_category}. Since $\cE'$ is actually stable, this forces $\Im u$ to be equal to $\cE'$. Then, again since $u$ is strict, $\Im u=\cE'$ implies that $u$ is surjective. Therefore, $u$ is an isomorphism.
\end{proof}

\noindent Note that a vector bundle always has non-trivial automorphisms (multiplication by a non-zero scalar on the fibres). When these are all the automorphisms of a given bundle, it is called a \textit{simple} bundle. We now show that stable implies simple.

\begin{proposition}\label{stable_implies_simple}
If $\cE$ is a stable vector bundle on $\Si_g$, then $\End\,\cE$ is a field, isomorphic to $\C$. In particular, $\Aut\,\cE\simeq\C^*$.
\end{proposition}

\begin{proof}
Let $u$ be a non-zero endomorphism of $\cE$. By Proposition \ref{isom_or_zero}, $u$ is an automorphism of $\cE$, so $\End\,\cE$ is a field, which contains $\C$ as its sub-field of scalar endomorphisms. Then, for any $u\in\End\,\cE$, the sub-field $\C(u)\subset\End\,\cE$ is a commutative field, and the Cayley-Hamilton Theorem shows that $u$ is algebraic over $\C$. Since $\C$ is algebraically closed, this shows that $u\in\C$. So $\End\,\cE\simeq\C$ (in particular, the field $\End\,\cE$ is commutative) and therefore $\Aut\,\cE\simeq\C^*$.
\end{proof}

\begin{corollary}
A stable vector bundle is indecomposable: it is not isomorphic to a direct sum of non-trivial sub-bundles.
\end{corollary}

\begin{proof}
The automorphism group of a direct sum $\cE=\cE_1\oplus\cE_2$ contains $\C^*\times\C^*$, so $\cE$ cannot be simple. Then, by Proposition \ref{stable_implies_simple}, it cannot be stable.
\end{proof}

The following result is key to understanding semi-stable bundles: these are extensions of stable bundles of the same slope.

\begin{theorem}[Seshadri, \cite{Seshadri}]
The simple objects in the category of semi-stable bundles of slope $\mu$ are the stable bundles of slope $\mu$. Any semi-stable holomorphic vector bundle of slope $\mu$ on $\Si_g$ admits a filtration $$0=\cE_0\subset \cE_1 \subset \cdots \subset \cE_k=\cE$$ by holomorphic sub-bundles such that, for all $i\in\{1;\cdots;k\}$, 
\begin{enumerate}
\item $\cE_i/\cE_{i-1}$ is stable,
\item $\mu(\cE_i/\cE_{i-1})=\mu(\cE)$. 
\end{enumerate} Such a filtration is called a \textbf{Jordan-H\"older filtration of length} $k$ of $\cE$.
\end{theorem}

\begin{proof}
Recall that a simple object in an Abelian category is an object with no non-trivial sub-object. In particular, a stable bundle $\cE$ is simple in that sense (it contains no non-trivial sub-bundle \textit{of slope equal to} $\mu(\cE)$). Conversely, if a semi-stable bundle $\cE$ is simple in that sense, then any non-trivial sub-bundle $\cF\subset \cE$ satisfies $\mu(\cF)\leq \mu(\cE)$ because $\cE$ is semi-stable, and $\mu(\cF)\neq \mu(\cE)$ because $\cE$ has no non-trivial sub-objects in the category of semi-stable bundles with slope $\mu$.\\ To prove the existence of a Jordan-H\"older filtration for a semi-stable bundle $\cE$, observe that increasing and decreasing sequences of sub-bundles of $\cE$ are stationary because of the bounds on the rank. If $\cE$ is not a simple object, there exists a non-trivial sub-bundle $\cE'$ of $\cE$ which is semi-stable and of slope $\mu$. If $\cE'$ is not a simple object, we can go on and find a decreasing sequence of non-trivial (semi-stable) sub-bundles (of slope $\mu$) in $\cE$. Such a sequence is stationary, and we call $\cE_1$ the final term: it is a simple sub-object of $\cE$, so it is a stable bundle of slope $\mu$. In particular, $\cE/\cE_1$ is semi-stable and also has slope $\mu$ (see Exercise \ref{invariance_under_extensions}). So there is a sub-bundle $\cE_2/\cE_1$ which is stable and of slope $\mu$. This gives an increasing sequence $$0=\cE_0\subset\cE_1 \subset \cE_2 \subset \cdots$$ of (semi-stable) sub-bundles of $\cE$ (of slope $\mu$) whose successive quotients are stable bundles of slope $\mu$. Such a sequence is stationary, so there is a $k$ such that $\cE_k=\cE$, and the resulting filtration of $\cE$ is a Jordan-H\"older filtration.
\end{proof}

\noindent One may observe that, to show the existence of a filtration whose successive quotients are simple objects in the category of semi-stable bundles of slope $\mu$, the proof only used that decreasing and increasing sequences \textit{of such bundles} were stationary. An Abelian category satisfying these properties is called \textit{Artinian} (decreasing sequences of sub-objects are stationary) and \textit{Noetherian} (increasing sequences of sub-objects are stationary).

Observe that if a bundle is stable, it admits a Jordan-H\"older filtration of length $1$, namely $0=\cE_0\subset \cE_1=\cE$. In general, there is no unicity of the Jordan-H\"older filtration, but the isomorphism class of the graded object associated to a filtration is unique, as shown by the next result. In particular, the lengths of any two Jordan-H\"older filtrations of $\cE$ are equal and a semi-stable bundle is stable if and only if its Jordan-H\"older filtrations have length $1$.

\begin{proposition}[Seshadri, \cite{Seshadri}]\label{graded_object}
Any two Jordan-H\"older filtrations $$(S): 0=\cE_0 \subset \cE_1 \subset \cdots \subset \cE_k=\cE$$ and $$(S'): 0=\cE'_0 \subset \cE'_1 \subset \cdots \subset \cE'_l=\cE$$ of a semi-stable vector bundle $\cE$ have same length $k=l$, and the associated graded objects $$\gr(S) := \cE_1/\cE_0 \oplus \cdots \oplus \cE_k/\cE_{k-1}$$ and $$\gr(S') := \cE'_1/\cE'_0 \oplus \cdots \oplus \cE'_k/\cE'_{k-1}$$ satisfy $$\cE_{i}/\cE_{i-1} \simeq \cE'_i/\cE'_{i-1}$$ for all $i\in\{1;\cdots;k\}$.
\end{proposition}

\begin{proof}
Assume for instance that $l<k$. Then there exists an $i\in\{1;\cdots;k\}$ such that $\cE_1'\subset\cE_i$ and $\cE'_1\not\subset \cE_{i-1}$. So the map $\cE´_1\hookrightarrow \cE_i  \lra \cE_i/\cE_{i-1}$ is a non-zero morphism between stable bundles of slope $\mu$. By Proposition \ref{isom_or_zero}, it is an isomorphism. So $\cE'_1\cap\cE_{i-1}=0$ and $\cE_i=\cE_{i-1}\oplus\cE'_1$. Then, $$(S_1): 0\subset \cE'_1 \subset \cE'_1\oplus \cE_1\subset \cdots \subset \cE'_1 \oplus\cE_{i-1} \subset \cE_{i+1} \subset \cdots\subset \cE_k=\cE$$ is a Jordan-H\"older filtration of length $k$ of $\cE$. Since $(S')$ and $(S_1)$ have the same first term, they induce Jordan-H\"older filtrations of $\cE/\cE'_1$, of respective lengths $l-1$ and $k-1$, with $l-1<k-1$. Repeating this process $l-1$ more times, we eventually reach $\cE/\cE'_{l-1}$ with a Jordan-H\"older filtration of length $k-l>0$. In particular, if the inclusions $\cE'_{l-1}\subset\cE_{k-1} \subset \cE_k=\cE$ are strict, there is a sub-bundle of $\cE_{k-1}/\cE'_{l-1}$ contradicting the stability of $\cE/\cE'_{l-1}$. So $l=k$.\\ Then we prove the second assertion by induction on the length $k$ of Jordan-H\"older filtrations of $\cE$. If $k=1$, it is obvious. If $k>1$, consider again the filtration $(S_1)$. It satisfies $\gr(S_1)\simeq\gr(S)$. Moreover, $(S_1)$ and $(S')$ have the same first term, so they induce Jordan-H\"older filtrations $(\overline{S_1})$ and $(\overline{S'})$ of length $k-1$ of $\cE/\cE'_{1}$. By the induction hypothesis $\gr(\overline{S_1})\simeq \gr(\overline{S'})$. So $$\gr(S)\simeq \gr(S_1) \simeq \gr(\overline{S_1}) \oplus \cE'_1 \simeq \gr(\overline{S'}) \oplus \cE'_1 \simeq \gr(S')\, .$$
\end{proof}

\noindent This motivates the following definition.

\begin{definition}[Poly-stable bundles]\label{pst_def}
A holomorphic vector bundle $\cE$ on $\Si_g$ is called poly-stable if it is isomorphic to a direct sum $\cF_1\oplus \cdots\oplus  \cF_k$ of stable vector bundles \textit{of the same slope}.
\end{definition}

\noindent Evidently, a stable bundle is poly-stable. We point out that a poly-stable vector bundle of rank $r$ admits a reduction of its structure group $\GL(r;\C)$ to a sub-group of the form $\GL(r_1;\C)\times \cdots \times \GL(r_k;\C)$, with $r_1+\cdots +r_k=r$ (also known as a Levi subgroup of $\GL(r;\C)$). An alternate definition of poly-stability is given as follows (also see Exercise \ref{poly_stability_again} for yet another characterization of poly-stability).

\begin{proposition}
A holomorphic vector bundle $\cE$ is poly-stable in the sense of Definition \ref{pst_def} if and only it is semi-stable and isomorphic to a direct sum of stable bundles.
\end{proposition}

\begin{proof}
The proof follows from Exercise \ref{invariance_under_extensions}.
\end{proof}

The graded object associated to any Jordan-H\"older filtration of a semi-stable vector bundle $\cE$ is a poly-stable vector bundle (since it is a direct sum of simple objects in the category of semi-stable bundles of slope $\mu(\cE)$, it is a semi-simple object in that category). By Proposition \ref{graded_object}, its isomorphism class is uniquely defined; it is usually denoted $\gr(\cE)$ and it is a graded isomorphism class of poly-stable vector bundles. The following notion is due to Seshadri (he used it to compactify the quasi-projective moduli variety of stable bundles constructed by Mumford).

\begin{definition}[$S$-equivalence class, \cite{Seshadri}]
The graded isomorphism class $\gr(\cE)$ associated to a semi-stable vector bundle $\cE$ is called the $S$\textbf{-equivalence class} of $\cE$. If $\gr(\cE)\simeq\gr(\cE')$, we say that $\cE$ and $\cE'$ are $S$-equivalent, and we write $\cE\sim_S\cE'$.
\end{definition}

\noindent This defines an equivalence relation between semi-stable bundles of a given fixed slope. If two bundles of slope $\mu$ are $S$-equivalent, they have the same rank and the same degree (because the rank and degree of $\cF_1\oplus \cdots \oplus \cF_k$ are equal to those of $\cE$, see Exercise \ref{rank_and_degree_of_a_polystable_bundle}). The important point is that \textit{two non-isomorphic semi-stable vector bundles may be $S$-equivalent}. Two $S$-equivalent stable bundles, however, are isomorphic, by definition of the $S$-equivalence class.

\begin{definition}[Moduli set of semi-stable vector bundles]
The set $\Mod(r,d)$ of $S$-equivalence classes of semi-stable holomorphic vector bundles of rank $r$ and degree $d$ on $\Si_g$ is called the \textbf{moduli set} of semi-stable vector bundles of rank $r$ and degree $d$. It contains the set $\Mods(r,d)$ of isomorphism classes of stable vector bundles of rank $r$ and degree $d$. When $r$ and $d$ are coprime, every semi-stable bundle is in fact stable and these two sets coincide.
\end{definition}

\noindent Equivalently, $\Mod(r,d)$ is the set of isomorphism classes of poly-stable holomorphic vector bundles of rank $r$ and degree $d$. This will be important in Subsection \ref{section_on_Donaldson_thm}, where Donaldson's Theorem will be presented. The following theorem is the main result of the basic theory of vector bundles on a curve. It is due to Mumford for the first part (\cite{Mumford_VB}) and Seshadri for the second part (\cite{Seshadri}).

\begin{theorem}[Mumford-Seshadri, \cite{Mumford_VB,Seshadri}]
Let $g\geq 2$, $r\geq 1$ and $d\in\Z$.
\begin{enumerate}
\item The set $\Mods(r,d)$ of isomorphism classes of stable holomorphic vector bundles of rank $r$ and degree $d$ admits a structure of smooth, complex quasi-projective variety of dimension $r^2(g-1) + 1$.
\item The set $\Mod(r,d)$ of $S$-equivalence classes of semi-stable holomorphic vector bundles of rank $r$ and degree $d$ admits a structure of complex projective variety of dimension $r^2(g-1)+1$. $\Mods(r,d)$ is an open dense sub-variety of $\Mod(r,d)$.
\end{enumerate}
\noindent In particular, when $r\wedge d=1$, $\Mod(r,d)=\Mods(r,d)$ is a smooth complex projective variety. 
\end{theorem}

\noindent For general $r$ and $d$, it can in fact be shown that the set of \textit{isomorphism} classes of semi-stable vector bundles of rank $r$ and degree $d$ does not admit such an algebraic structure (\cite{Seshadri_ast}). In other words, to obtain a moduli variety, we \textit{have to} identify $S$-equivalent, possibly non-isomorphic, objects.

\subsection{The Harder-Narasimhan filtration}

We conclude this section on stability of holomorphic vector bundles on a curve by showing the existence of a canonical filtration for \textit{any} holomorphic vector bundle, called the Harder-Narasimhan filtration.

\begin{theorem}[Harder-Narasimhan, \cite{HN}]\label{HN_filtration}
Any holomorphic vector bundle $\cE$ on $\Si_g$ has a unique filtration $$0=\cE_0\subset \cE_1\subset \cdots \subset \cE_l=\cE$$ by holomorphic sub-bundles such that
\begin{enumerate}
\item for all $i\in\{1;\cdots;l\}$, $\cE_i/\cE_{i-1}$ is semi-stable,
\item the slopes $\mu_i:=\mu(\cE_i/\cE_{i-1})$ of the successive quotients satisfy $$\mu_1 > \mu_2 > \cdots >\mu_l\, .$$
\end{enumerate}
This filtration is called the \textbf{Harder-Narasimhan filtration}.
\end{theorem}

\noindent Before giving the proof of the theorem, let us make a few comments. The importance of the Harder-Narasimhan filtration cannot be stressed enough. It was instrumental, for instance, in the work of Atiyah and Bott (\cite{AB}) on the topology of the moduli space $\Mod(r,d)$. More generally, and to paraphrase if I may Robert Friedman in \cite{Friedman}, the Harder-Narasimhan filtration is how we understand holomorphic bundles: they have a unique filtration whose successive quotients are semi-stable bundles with slopes in decreasing order. In turn, Jordan-H\"older filtrations are how we understand semi-stable bundles: they are extensions of stable bundles of the same slope (necessarily equal to the slope of the bundle). In contrast, we do not quite understand stable bundles, unless, perhaps, from the point of view of the Narasimhan-Seshadri Theorem: they correspond to irreducible projective unitary representations of the fundamental groups of $\Si_g$.

If we denote $r=\rk\cE$, $d=\deg\cE$, $r_i=\rk(\cE_i/\cE_{i-1})$ and $d_i=\deg(\cE_i/\cE_{i-1})$, we have $$r_1+\cdots +r_l =r\quad \mathrm{and}\quad d_1+\cdots +d_l=d\,.$$ The $r$-tuple $$\umu:=(\underbrace{\mu_1,\cdots,\mu_1}_{r_1\ \mathrm{times}}, \cdots, \underbrace{\mu_l,\cdots,\mu_l}_{r_l\ \mathrm{times}})$$ is called the \textbf{Harder-Narasimhan type of} $\cE$. It is equivalent to the data of the $l$-tuple $(r_i,d_i)_{1\leq i\leq l}$. In the plane of coordinates $(r,d)$, the polygonal line $$P_{\umu}:=\{(0,0),(r_1,d_1),(r_1+r_2,d_1+d_2),\cdots,(r_1+\cdots+r_l,d_1+\cdots+d_l)\}$$ defines a convex polygon called the Harder-Narasimhan, or Shatz, polygon of $\cE$ (see \cite{Shatz}): the slope of the line from $(r_1,d_1)$ to $(r_1+r_2,d_1+d_2)$ is $\frac{d_2}{r_2} = \mu(\cE_2/\cE_1)$, i.e. the slope in the algebro-geometric sense of the bundle $\cE_2/\cE_1$ (hence, perhaps, the terminology), and likewise for higher indices; since $\mu_1>\cdots>\mu_l$, the Shatz polygon is indeed convex. The vector bundle $\cE$ is semi-stable if and only if it is its own Harder-Narasimhan filtration, i.e. if and only if its Shatz polygon is a single straight line going from $(0,0)$ to $(r,d)$. Let us now move on to the proof of Theorem \ref{HN_filtration}.

\begin{lemma}\label{slope_is_bounded}
Let $\cE$ be a non-zero holomorphic vector bundle on $\Si_g$ and set $$\mu_{\max}(\cE) := \sup\, \{\mu(\cF)\, :\, \cF\subset \cE,\, \cF\neq 0\}.$$ Then $\mu_{\max}(\cE)<+\infty$.
\end{lemma}

\begin{proof}
The proof we give requires algebro-geometric notions and may be skipped upon first reading of these notes. Let $\cL$ be a very ample line bundle on $\Si_g$. Then, by the Cartan-Serre Theorem, $\cE^*\otimes\cL$ is generated by its global sections, so $\cE$ may be seen as a sub-bundle of $\underline{\C^N}\otimes\cL$ for some $N\in\mathbb{N}$, where $\underline{\C^N}$ denotes the product bundle of rank $N$ on $\Si_g$. Since this product bundle is a semi-stable (in fact, poly-stable) bundle and since tensoring by a line bundle preserves semi-stability (see Exercise \ref{tensorisation_by_line_bundle}), $\underline{\C^N}\otimes\cL$ is semi-stable of slope $$\mu(\underline{\C^N}\otimes \cL) =\frac{(\deg\underline{\C^N})\,\rk\cL + (\rk\underline{\C^N})\,\deg\cL}{(\rk\underline{\C^N})\,(\rk\cL)} = \deg\cL\, ,$$ as $\deg\underline{\C^N}=0$ and $\rk\cL=1$. So, for any non-zero sub-bundle $\cF\subset\cE\subset\underline{\C^N}\otimes\cL$, $$\mu(\cF) \leq \mu(\underline{\C^N}\otimes \cL) = \deg\cL\, $$ which shows that $\mu_{\max}(\cE)\leq\deg\cL<+\infty$.
\end{proof}

\begin{lemma}\label{towards_unicity}
Let $\cE$ be a semi-stable bundle and let $\cE'$ be any holomorphic bundle on $\Si_g$. If $\mu(\cE)>\mu_{\max}(\cE')$, then any homomorphism $u:\cE\lra\cE'$ is zero.
\end{lemma}

\begin{proof}
Assume $\Im{u}\neq 0$. By Lemma \ref{slopes_of_image_and_coimage}, $\mu(\Im{u}) \geq \mu(\cE/\ker{u})$. Since $\cE$ is semi-stable, $\mu(\cE/\ker{u})\geq \mu(\cE)> \mu_{\max}(\cE')$ , so $\mu(\Im{u})>\mu_{\max}(\cE')$, a contradiction.
\end{proof}

\begin{lemma}\label{towards_existence_of_destabilising_bundle}
Let $\cE$ be any holomorphic bundle on $\Si_g$. Then there exist sub-bundles of $\cE$ of slope $\mu_{\max}(\cE)$.
\end{lemma}

\begin{proof}
By Lemma \ref{slope_is_bounded}, the set $$\cM(\cE):=\{\mu(\cF)\, :\, \cF\subset\cE,\, \cF\neq 0\}$$ is bounded from above. Since the possible slopes of sub-bundles of $\cE$ lie in the set $$\left\{\frac{d'}{r'}\, :\, d'\in\Z,\, 1\leq r\leq\rk\cE\right\}\, ,$$ $\cM(\cE)$ is a discrete subset of $\R$. So $\mu_{\max}(\cE) = \sup\cM(\cE)$ is attained: there are sub-bundles $\cF\subset\cE$ such that $\mu(\cF)=\mu_{\max}(\cE)$. 
\end{proof}

\noindent Since, moreover, the set $$\{\rk\cF\, :\, \cF\in\cM(\cE)\}$$ is finite, the notion of sub-bundle of maximal rank in $\cM(\cE)$ is well-defined.

\begin{lemma}\label{towards_unicity_of_destabilising_bundle}
Let $\cE$ be any holomorphic vector bundle on $\Si_g$ and let $\cE_1$ be a sub-bundle of maximal rank among sub-bundles of slope $\mu_{\max}(\cE)$. Then:
\begin{enumerate}
\item $\cE_1$ is semi-stable and $\mu_{\max}(\cE) > \mu_{\max}(\cE/\cE_1)$,
\item if $\cE'_1$ is another sub-bundle of maximal rank among sub-bundles of maximal slope of $\cE$, then $\cE'_1=\cE_1$.
\end{enumerate}
\end{lemma}

\begin{proof}
This will lead later to the notion of \textit{destabilising bundle} (Proposition \ref{unique_destabilising_bundle}).
\begin{enumerate}
\item If $\cE_1$ is a sub-bundle satisfying the assumptions of the Lemma, it is semi-stable because a sub-bundle $\cF$ of $\cE_1$ is also a sub-bundle of $\cE$, so it satisfies $\mu(\cF)\leq\mu_{\max}(\cE) = \mu(\cE_1)$. Consider now a sub-bundle $\cE_2$ of $\cE$, strictly containing $\cE_1$ and such that $\mu(\cE_2/\cE_1) = \mu_{\max}(\cE/\cE_1)$. Since the sequence $$0\lra \cE_1 \lra \cE_2 \lra \cE_2/\cE_1 \lra 0$$ is exact, $\mu(\cE_1) > \mu(\cE_2/\cE_1)$ if and only if $\mu(\cE_1) > \mu(\cE_2)$ (see Exercise \ref{invariance_under_extensions}). But $\cE_2$ is a sub-bundle of $\cE$, so $\mu(\cE_2)\leq \mu_{\max}(\cE)=\mu(\cE_1)$. Since $\mu(\cE_2)=\mu(\cE_1)$ would contradict the maximality of $\rk\cE_1$ for sub-bundles of $\cE$ having slope $\mu_{\max}(\cE)$, one has $\mu(\cE_1)>\mu(\cE_2)$, which implies that $\mu(\cE_1)>\mu(\cE_2/\cE_1)$, i.e. $\mu_{\max}(\cE) > \mu_{\max}(\cE/\cE_1)$. Note that, here, it is possible to have $\cE_2=\cE$ (this is what happens if $\cE/\cE_1$ is semi-stable).
\item Consider the composed map $$\cE'_1 \hookrightarrow \cE \lra \cE/\cE_1\, .$$ Since, by (1), $\cE'_1$ is also semi-stable and satisfies $\mu(\cE'_1) = \mu_{\max}(\cE) > \mu_{\max}(\cE/\cE_1)$, Lemma \ref{towards_unicity} shows that this composed map is zero. So $\cE'_1\subset \cE_1$, therefore $\cE'_1 =\cE_1$ since $\rk\cE'_1=\rk\cE_1$.
\end{enumerate}
\end{proof}

\begin{proposition}\label{unique_destabilising_bundle}
There exists a unique sub-bundle of $\cE$ whose rank is maximal among sub-bundles of maximal slope of $\cE$. It is called the \textbf{destabilising bundle} of $\cE$, and we shall denote it $\mathrm{G}(\cE)$.
\end{proposition}

\begin{proof}
The existence follows from Lemma \ref{towards_existence_of_destabilising_bundle} and the unicity follows from point (2) of Lemma \ref{towards_unicity_of_destabilising_bundle}.
\end{proof}

\noindent Oddly enough for this standard terminology, if $\cE$ is semi-stable, it is its own destabilising bundle, and the converse also holds. The next result is a characterisation of the destabilising bundle which is a converse to Point (1) of Lemma \ref{towards_unicity_of_destabilising_bundle}.

\begin{proposition}\label{charac_of_destabilising_bundle}
If $\cE_1$ is a sub-bundle of $\cE$ such that $\cE_1$ is semi-stable and satisfies $\mu_{\max}(\cE)>\mu_{\max}(\cE/\cE_1)$, then $\cE_1$ is the destabilising bundle of $\cE$.
\end{proposition}

\begin{proof}
Consider the composed map $$\mathrm{G}(\cE) \hookrightarrow \cE \lra \cE/\cE_1\, ,$$ where $\mathrm{G}(\cE)$ is the destabilising bundle of $\cE$. By Lemma \ref{towards_unicity_of_destabilising_bundle}, $\mathrm{G}(\cE)$ is semi-stable and satisfies $\mu(\mathrm{G}(\cE)) = \mu_{\max}(\cE)> \mu_{\max}(\cE/\cE_1)$. So, by Lemma \ref{towards_unicity}, this map is zero and therefore $\mathrm{G}(\cE)\subset\cE_1$. Since $\cE_1$ is semi-stable, this implies that $\mu(\mathrm{G}(\cE))\leq \mu(\cE_1)$, so $\mu(\mathrm{G}(\cE))=\mu(\cE_1)$, since $\mu(\mathrm{G}(\cE))=\mu_{\max}(\cE)$. But $\mathrm{G}(\cE)\subset\cE_1$ so $\mathrm{G}(\cE)=\cE_1$, by maximality of $\mathrm{G}(\cE)$ among sub-bundles of maximal slope of $\cE$.
\end{proof}

\noindent We are now in a position to prove the Harder-Narasimhan Theorem.

\begin{proof}[Proof of Theorem \ref{HN_filtration}]
We first prove the unicity of a filtration $$0=\cE_0\subset\cE_1\subset \cdots \subset \cE_l=\cE$$ satisfying $\cE_i/\cE_{i-1}$ semi-stable for all $i\geq 1$ and $$\mu(\cE_1/\cE_0) > \cdots > \mu(\cE_l/\cE_{l-1})$$ (a Harder-Narasimhan filtration) by showing that, for all $i\geq 1$, $\cE_{i}/\cE_{i-1}$ is the destabilising bundle of $\cE/\cE_{i-1}$. The proof is by induction on the rank of $\cE$. If $\rk\cE=1$, the only possible filtration of $\cE$ is $0=\cE_0\subset\cE_1=\cE$, and since a bundle of rank $1$ is semi-stable, $\cE_1=\cE$ is indeed the destabilising bundle of $\cE/\cE_0=\cE$. Let now $$0=\cE_0\subset\cE_1\subset \cdots \subset \cE_l=\cE$$ be a Harder-Narasimhan filtration in rank $r$. Then $$0=\cE_1/\cE_1\subset \cE_2/\cE_1\subset \cdots \subset \cE_l/\cE_1=\cE/\cE_1$$ is a Harder-Narasimhan filtration in rank $<r$ so, by the induction hypothesis, $$\left(\cE_i/\cE_{1}\right) \big/ \left(\cE_{i-1}/\cE_1\right) \simeq \cE_i/\cE_{i-1}$$ is the destabilising bundle of $$\left(\cE/\cE_{1}\right) \big/ \left(\cE_{i-1}/\cE_1\right) \simeq \cE/\cE_{i-1}\, .$$ It remains to show that $\cE_1$ is the destabilising bundle of $\cE$. Since we have just seen that $\cE_2/\cE_1 = \mathrm{G}(\cE/\cE_1)$, we have $$\mu_{\max}(\cE/\cE_1) = \mu(\cE_2/\cE_1) < \mu(\cE_1) \leq \mu_{\max}(\cE)$$ because $\cE_0\subset\cE_1 \subset\cdots\subset \cE_l$ is, by assumption, a Harder-Narasimhan filtration. For the same reason, we have, moreover, that $\cE_1$ is semi-stable. So Lemma \ref{charac_of_destabilising_bundle} shows that $\cE_1=\mathrm{G}(\cE)$.

To prove the existence of the Harder-Narasimhan filtration, we proceed again by induction on the rank, the result being obvious in rank $1$. If $\cE$ has rank $r$, set $\cE_1=\mathrm{G}(\cE)$, the destabilising bundle of $\cE$. Then, by Lemma \ref{towards_unicity_of_destabilising_bundle}, $\cE_1$ is semi-stable and $\cE/\cE_1$ is a bundle of rank $<r$ satisfying $\mu_{\max}(\cE/\cE_1) < \mu_{\max}(\cE) = \mu(\cE_1)$. By the induction hypothesis, $\cE/\cE_1$ has a Harder-Narasimhan filtration, which pulls back to a filtration $$\cE_1\subset \cE_2 \subset \cdots \cE_l=\cE$$ such that, for $i\geq 2$, $\cE_i/\cE_{i-1}$ is semi-stable and $$\mu(\cE_2/\cE_1) > \cdots > \mu(\cE_l/\cE_{l-1})\, .$$ Since $\mu(\cE_2/\cE_1) \leq \mu_{\max}(\cE/\cE_1)$ (in fact, it is an equality since $\cE_2/\cE_1=\mathrm{G}(\cE/\cE_1)$), one has $\mu(\cE_2/\cE_1) < \mu(\cE_1)$ so $$0=\cE_0\subset\cE_1\subset \cdots \subset \cE_l=\cE$$ is a Harder-Narasimhan filtration of $\cE$.
\end{proof}

The Harder-Narasimhan filtration was used by Atiyah and Bott in \cite{AB} to define a stratification of the space of all holomorphic structures (Dolbeault operators) on a smooth complex vector bundle $E$ of rank $r$ and degree $d$, which they showed to be equivariantly perfect for the action of the complex gauge group $\cG_E$. Two holomorphic structures belong to the same stratum $\mathcal{C}_{\umu}$ if and only if the holomorphic bundles that they define have the same Harder-Narasimhan type $\umu$. The Harder-Narasimhan strata have finite codimension and are $\cG_E$-invariant. The semi-stable bundles form the unique open stratum of this stratification and, as a consequence of equivariant perfection, the $\cG_E$-equivariant cohomology of this stratum may be computed. When $r\wedge d=1$, one can deduce from this computation, among other topological information, the (rational) Betti numbers of $\Mod(r,d)$ (see \cite{AB}). It was later shown by Daskalopoulos that the Harder-Narasimhan strata were the Morse strata of the Yang-Mills functional (\cite{Dask}), confirming a conjecture of Atiyah and Bott: the Yang-Mills functional is a Morse-Bott function whose Morse flow converges and whose critical manifolds consist of bundles which are direct sums of the form $$\cE'_1\oplus \cdots \oplus \cE'_l\,$$ where $\cE'_i$ is a \textit{poly}-stable bundle of rank $r_i$ and degree $d_i$ satisfying $$\frac{d_1}{r_1} > \cdots > \frac{d_l}{r_l}\, ,$$ as well as $r_1+\cdots + r_l=r$ and $d_1+\cdots+d_l=d$. The absolute minima of the Yang-Mills functional are poly-stable bundles of rank $r$ and degree $d$ (that is, critical points of the form above satisfying the additional condition that $l=1$). In particular, the Morse flow of the Yang-Mills functional takes a semi-stable bundle $\cE$ to the graded object $\gr(\cE)$ associated to any Jordan-H\"older filtration of $\cE$, and the latter is an absolute minimum of the Yang-Mills functional. In general, the Morse flow takes $\cE$ to $\cE'_1\oplus \cdots \oplus \cE'_l$, where $\cE'_i$ is the graded object associated to the semi-stable bundle $\cE_i/\cE_{i-1}$ provided by the Harder-Narasimhan filtration of $\cE$ (in particular, $\rk\cE'_i=r_i$ and $\deg\cE'_i=d_i$ where $r_i=\rk(\cE_i/\cE_{i-1})$ and $d_i=\deg(\cE_i/\cE_{i-1})$, so $(r_i,d_i)_{1\leq i\leq l}$ corresponds to the Harder-Narasimhan type $\umu$ of $\cE$). This graded object is indeed a critical point of the Yang-Mills functional.

\subsection{Exercises}\hfill

\begin{multicols}{2}

\begin{exercise}
Show that a semi-stable holomorphic which has coprime rank and degree is in fact stable.
\end{exercise}

\begin{exercise}
Show that $\mu(\cE^*) = -\mu(\cE)$ and $\mu(\cE\otimes\cE') = \mu(\cE) + \mu(\cE')$. Compute $\mu(\Hom(\cE,\cE'))$.
\end{exercise}

\begin{exercise}\label{invariance_under_extensions}
Consider the extension (short exact sequence) $$0\lra \cE' \lra \cE \lra \cE''\lra 0$$ of $\cE''$ by $\cE'$.\\ \textbf{a.} Assume that $\cE'$ and $\cE''$ are semi-stable and both have slope $\mu$. Show that $\mu(\cE)=\mu$ and that $\cE$ is semi-stable.\\ \textbf{b.} Show that if $\cE'$ and $\cE''$ are stable and have the same slope, $\cE$ is not stable. \textit{Hint}: By \textbf{a}, $\mu(\cE)=\mu(\cE')$ and $\cE'$ is a sub-bundle of $\cE$.\\ \textbf{c.} Let $\mu$, $\mu'$, $\mu''$ be the respective slopes of the bundles $\cE$, $\cE'$, $\cE''$. Show that $$\mu' < \mu \Leftrightarrow \mu' < \mu'' \Leftrightarrow \mu < \mu''\, ,$$ $$\mu' = \mu \Leftrightarrow \mu' = \mu'' \Leftrightarrow \mu = \mu''\, ,$$ $$\mu' > \mu \Leftrightarrow \mu' > \mu'' \Leftrightarrow \mu > \mu''\, .$$ \textbf{c.} Suppose that the three bundles $\cE$, $\cE'$ and $\cE''$ have the same slope. Show that $\cE$ is semi-stable if and only if $\cE'$ and $\cE''$ are semi-stable.
\end{exercise}

\begin{exercise}\label{rank_and_degree_of_a_polystable_bundle}
Let $\cE$ and $\cE'$ be two semi-stable bundles of slope $\mu$ and assume that $\cE$ and $\cE'$ are $S$-equivalent. Show that $\rk\cE=\rk\cE'$ and $\deg\cE=\deg\cE'$. \textit{Hint}: Consider the poly-stable object $\cF_1\oplus \cdots \oplus \cF_k$ associated to an arbitrary Jordan-H\"older filtration of $\cE$, and show that $\deg(\cF_1) + \cdots + \deg(\cF_k) = \deg(\cE)$. Beware that the direct sum $\cF_1\oplus \cdots \oplus \cF_k$ is not isomorphic to $\cE$ in general.
\end{exercise}

\begin{exercise}\label{tensorisation_by_line_bundle}
Let $\cE$ be a holomorphic vector bundle and let $\cL$ be a holomorphic line bundle.\\ \textbf{a.} Show that $\mu(\cE\otimes\cL)=\mu(\cE)+\mu(\cL)$.\\ \textbf{b.} Show that $\cE$ is stable (resp. semi-stable) if and only if $\cE\otimes\cL$ is stable (resp. semi-stable). \textit{Hint}: Sub-bundles of $\cE\otimes\cL$ are of the form $\cF\otimes\cL$, where $\cF$ is a sub-bundle of $\cE$.
\end{exercise}

\begin{exercise}\label{poly_stability_again}
Show that a vector bundle $\cE$ is poly-stable if and only if it is semi-stable and admits a reduction of its structure group $\GL(r;\C)$ to a sub-group $L$ of the form $\GL(r_1;\C)\times \cdots \times \GL(r_k;\C)$ with $r_1+\cdots +r_k=r$ such that the associated $L$-vector bundle is stable (that is, with respect to sub-bundles which are also direct sums). 
\end{exercise}

\end{multicols}

\section{The moduli variety as a K\"ahler quotient}\label{YM_connections}

\subsection{K\"ahler reduction}\label{kahler_red}

In the next subsection, we shall see how Donaldson's theorem implies that the moduli space $\Mod(r,d)$, of $S$-equivalence classes of semi-stable vector bundles of rank $r$ and degree $d$ on $\Si_g$, is a K\"ahler quotient, so we briefly recall the theory of such quotients (see for instance \cite{McDuff,HKLR}).

\begin{definition}[K\"ahler manifold]
A K\"ahler manifold $(M,J,g,\w)$ is a complex analytic manifold $(M,J)$ endowed with a Riemannian metric $g$, such that $\w := g(J\,\cdot\,,\,\cdot\,)$ is a symplectic form.
\end{definition}
 
 \noindent A less condensed definition would be as follows. We think of a complex analytic manifold $M$ as an even-dimensional real manifold with an \textit{integrable} almost complex structure $J$ (an almost complex structure being, by definition, an endomorphism of $TM$ squaring to $-\Id_{TM}$, which can only happen if $M$ is even-dimensional). Then the Riemannian manifold $(M,g)$ is called K\"ahler if

\begin{enumerate}
\item $J$ is an isometry for $g$: $g(Jv,Jw) = g(v,w)$,
\item the associated non-degenerate $2$-form $\w:=g(J\,\cdot\,,\,\cdot\,)$ is \textit{closed}.
\end{enumerate}

\noindent Then the metric $g$ locally \textit{derives from a potential} and the complex structure commutes to the covariant derivative of the Levi-Civit\`a connection of the metric. Note that, if $\dim_{\R}M=2$ and condition (1) is satisfied, then condition (2) is necessarily satisfied, 

An action of a Lie group is called a K\"ahler action if it preserves $g$, $J$ and $\w$. As a matter of fact, it suffices to preserve two of those to preserve the third one. We now consider  the case of a \textit{Hamiltonian} action of a \textit{compact} connected Lie group $\bK$ on a K\"ahler manifold $(M,J,g,\w)$, i.e. a K\"ahler action for which there is an equivariant momentum map $$\mu: M \lra \fk^*.$$ We recall that the fundamental vector field associated to an element $X\in\fk$ is defined by $$X^{\#}_x = \frac{d}{dt}|_{t=0} \big( \exp(tX)\cdot x\big)$$ (so the map $\chi: X \lmt X^{\#}$ is a homomorphism of Lie algebras $\fk\lra \Ga(TM)$ when the Lie bracket on $\fk$ is defined by means of the bracket of the corresponding \textit{right}-invariant vector fields, see \cite{McDuff}, Remark 3.3) and that the momentum map relation is $$\w(X^{\#},\,\cdot\,) = d<\mu,X> =<T\mu,X>$$ for all $X\in\fk$ (where $<\mu,X>$ is the function defined on $M$ by $x\lmt <\mu(x),X>$). In particular, if the action is Hamiltonian, then there is a map $$\mu^{\#}: \begin{array}{ccc} \fk & \lra & C^{\infty}(M;\R)\\ X & \lmt &<\mu,X>\end{array}$$ (sometimes called the co-momentum map), which is a Lie algebra homomorphism with respect to the Poisson bracket $$\{f,g\} := \w(\nabla^{\mathrm{symp}}_f,\nabla^{\mathrm{symp}}_g)$$ on $C^{\infty}(M;\R)$ and lifts $\chi$ to $C^{\infty}(M;\R)$ in the following sense: the diagramme

$$
\xymatrix{
&  C^{\infty}(M;\R) \ar[d]\\
\fk \ar@{-->}[ur]^{\mu^{\#}} \ar[r]^{\chi} & \Ga(TM)
}
$$

\noindent where the vertical map is the Lie algebra homomorphism taking a function $f:M\lra \R$ to the associated Hamiltonian vector field $\nabla^{\mathrm{symp}}_f$ defined by $\w(\nabla^{\mathrm{symp}}_f,\,\cdot\,) = df$, is a commutative diagramme. Recall that $\bK$ acts on $\fk^*$ by the co-adjoint action $\Adet{k} \xi = \xi \circ \Ad{k^{-1}}$, and that the centre of $\fk^*$ is the set of elements $\xi\in\fk^*$ on which $\bK$ acts trivially (i.e. $\mathrm{Ad}$-invariant linear forms on $\fk$). The main result of this subsection is as follows.

\begin{theorem}[K\"ahler reduction]\label{thm:Kahler_red}
Let $\bK$ be a compact Lie group acting on the K\"ahler manifold $(M,J,g,\w)$ with equivariant momentum map $\mu:M\lra\fk^*$. Let $\xi$ be an element of the centre of $\fk^*$, and assume that the action of $\bK$ on the level set $\mu^{-1}(\{\xi\})$ is free.
We denote $i:\mu^{-1}(\{\xi\})\hookrightarrow M$ the canonical inclusion. Then:
\begin{enumerate}
\item $\mu^{-1}(\{\xi\})$ is a submanifold of $M$, upon which $\bK$ acts with smooth quotient, and the map $$p: \mu^{-1}(\{\xi\}) \lra \mu^{-1}(\{\xi\})/\bK$$ is a principal fibration of group $\bK$.
\item The $2$-form $i^*\w$ on $\mu^{-1}(\{\xi\})$ is basic with respect to the projection $p$.
\item The corresponding $2$-form $\wred$ on $\mu^{-1}(\{\xi\})/\bK$, defined by $p^*\wred=i^*\w$, is a symplectic form.
\item The Riemannian metric $g$ and the complex structure $J$ on $M$ induce a Riemannian metric and a compatible almost complex structure on $\mu^{-1}(\{\xi\})/\bK$. The associated $2$-form is $\wred$, which is symplectic.
\item $\mu^{-1}(\{\xi\})/\bK$ is a K\"ahler manifold with respect to these induced metric and almost complex structure.
\end{enumerate}
$\mu^{-1}(\{\xi\})/\bK$ is called the K\"ahler quotient of $M$ at $\xi$.
\end{theorem}

\noindent Before giving a proof of this theorem, we point out that we can assume that $\xi=0$. Indeed, if $\xi\neq0$, then $\mu':=\mu-\xi$ is also an equivariant momentum map (because $\Adet{k}\xi = \xi$ for all $k\in \bK$), and $\mu^{-1}(\{\xi\}) = (\mu')^{-1}(\{0\})$.

If $x\in M$, we denote $$\bK_x = \{k\in \bK\ |\ k\cdot x = x\}$$ the stabiliser of $x$ in $\bK$, and $\fk_x=\mathrm{Lie}(\bK_x)$ its Lie algebra. Then $$\fk_x=\{X\in\fk\ |\ X^{\#}_x=0\}.$$ The anihilator of $\fk_x\subset \fk$ in $\fk^*$ is the vector space $$\fk_X^{\, 0}=\{\xi\in\fk^*\ |\ \forall X\in\fk_x,\, \xi(X)=0\}\subset\fk^*.$$ Recall that the tangent space at $x$ to the orbit $\bK\cdot x\subset M$ is $$T_x(\bK\cdot x) = \{X^{\#}_x\, :\, X\in\fk\}.$$ We denote $$\big(T_x(\bK\cdot x)\big)^{\w_x} = \{v\in T_xM\ |\ \forall w \in T_x(\bK \cdot x), \w_x(v,w)=0\}$$ the symplectic complement to $T_x(\bK\cdot x)$ in $T_xM$.

\begin{lemma}\label{kernel_and_image}
Let $x$ be a point in $\fibre$. Then
\begin{enumerate}
\item $\ker T_x\mu = (T_x(\bK\cdot x))^{\w_x}$.
\item $\Im T_x\mu = \fk_x^{\, 0}$.
\end{enumerate}
\end{lemma}

\begin{proof}
Take $v$ in $T_xM$.
\begin{enumerate}
\item $T_x\mu\cdot v =0$ in $\fk^*$ if and only if, for all $X\in\fk$, $<T_x\mu\cdot v,X> =0$, i.e. $\w_x(X^{\#}_x,v)=0$. Since any tangent vector to $\bK\cdot x$ is the value of a fundamental vector field, this is equivalent to $v\in (T_x(\bK\cdot x))^{\w_x}$.
\item Let $\xi:=T_X\mu\cdot v$. Then, for all $X\in\fk_x$, $$<\xi,X>\ =\ <T_x\mu\cdot v,X>\ =\ \w_x(X^{\#},v)\ =\ \w(0,v)\ =\ 0$$ so $\Im T_x\mu \subset \fk_x^{\, 0}$. The equality follows by dimension count, using (1). 
\end{enumerate}
\end{proof}

\begin{lemma}\label{infinetisimal_complex_action}
Let $x$ be a point in $\fibre$. Then
\begin{enumerate}
\item $\ker T_x\mu$ is a co-isotropic subspace of $T_xM$, meaning that $$(\ker T_x\mu)^{\w_x} \subset \ker T_x\mu.$$
\item Let $H_x$ denote the orthogonal complement to $(\ker T_x\mu)^{\w_x}$ in $\ker T_x\mu$ (with respect to the induced Riemannian metric on $\ker T_x\mu\subset T_x M$). Then there is a direct sum decomposition $$\underbrace{H_x\oplus T_x(\bK\cdot x)}_{\ker T_x\mu}\ \oplus\ J\, T_x(\bK\cdot x) = T_xM.$$ In particular, $H_x$ is $J$-invariant, so it is a complex vector space.
\end{enumerate}
\end{lemma}

\begin{proof}
Recall from Lemma \ref{kernel_and_image}, that $\ker T_x\mu = (T_x(\bK\cdot x))^{\w_x}$. Since $\w_x$ is non-degenerate, this implies that $(\ker T_X\mu)^{\w_x}= T_x(\bK\cdot x)$.
\begin{enumerate}
\item We want to show that $T_x(\bK\cdot x) \subset \ker T_x\mu$. Let $v=X^{\#}_x\in T_x(\bK\cdot x)$. Then, by definition,

\begin{eqnarray*}
T_x\mu\cdot X^{\#}_x & = & \frac{d}{dt}|_{t=0} \left(\mu(\exp(tX)\cdot x)\right) \\
& = & \frac{d}{dt}|_{t=0} \left(\Adet{\exp(tX)}\mu(x)\right) \\
& = & 0
\end{eqnarray*}

\noindent since $\mu$ is equivariant and $\mu(x)$ lies in the centre of $\fk$ so is acted upon trivially by $\exp(tX)$.
\item We want to show that $$\ker T_x\mu \oplus J\, T_x(\bK\cdot x) = T_x M.$$ Since the dimensions match, we need only show that $\ker T_x\mu \cap JT_x(\bK\cdot x) =\{0\}$. Let  $v=J X^{\#}_x \in \ker T_x\mu \oplus J T_x(\bK\cdot x)$. Then

\begin{eqnarray*}
g(v,v) & = & g(JX^{\#}_x, J X^{\#}_x) \\
& = & \w(X^{\#}_x,JX^{\#}_x) \\
& = & <T_x\mu\cdot JX^{\#}_x,X> \\
& = & <0,X>=0.
\end{eqnarray*}

\noindent Since $g$ is positive definite, this implies that $v=0$.
\end{enumerate}
\end{proof}

\noindent We are now in a position to prove Theorem \ref{thm:Kahler_red}.

\begin{proof}[Proof of Theorem \ref{thm:Kahler_red}]
As noted before, we may assume that $\xi=0$.
\begin{enumerate}
\item Observe that $\bK$ acts on $\fibre$ because $\mu$ is equivariant and $0$ is $\bK$-invariant in $\fk^*$. By Lemma \ref{kernel_and_image}, $\Im T_x\mu=\fk_x^{\, 0}$. Since $\bK$ acts freely on $\fibre$, we have $\fk_x=\{0\}$ for all $x\in\fibre$. So $T_x\mu$ is surjective and, by the submersion theorem, $\fibre$ is a submanifold of $M$. Since $\bK$ is compact and acts freely on the manifold $\fibre$, the topological quotient $\fibre/\bK$ is a manifold and $p:\fibre\lra  \fibre/\bK$ is a principal fibration of group $\bK$.
\item Note that the tangent space at $x$ to the manifold $\fibre$ is $\ker T_x\mu$. Saying that $i^*\w$ is basic with respect to $p$ means that
\begin{enumerate}
\item $L_{X^{\#}}(i^*\w) = 0$ for all $X$ in $\fk$,
\item $(i^*\w)(X^{\#},\,\cdot\,) = 0$ for all $X$ in $\fk$.
\end{enumerate}
The first condition follows from the fact that the action of $\bK$ preserves $\w$. The second condition is a consequence of the fact, proved in Lemma \ref{kernel_and_image}, that $T_x(\bK\cdot x) = (\ker T_x\mu)^{\w_x}$.
\item Since $i^*\w$ is basic with respect to $p$, there exists a unique $2$-form $\wred$ on $\fibre/\bK$ satisfying $p^*\wred=i^*\w$. Explicitly, it is defined by

\begin{equation}\label{def_omega_red}
\wred_{[x]}([v],[w]) = \w_x(v,w)
\end{equation}

\noindent (see Exercise \ref{check_def_omega_red} for details). It is closed because $$p^*(d\wred) = d(p^*\wred) = d(i^*\w) = i^*(d\w) =0$$ (and the exterior differential of a basic form is a basic form). The kernel of $\wred$ at $[x]\in \fibre/\bK$ is in bijection with $$(\ker T_x\mu)^{\w_x} / T_x(\bK\cdot x)\, ,$$ but this space is trivial by Lemma \ref{kernel_and_image}. So $\wred$ is non-degenerate.
\item In the notation of Lemma \ref{infinetisimal_complex_action}, the tangent space at $[x]$ to $\fibre/\bK$ is in bijection with $H_x\subset T_xM$, which is a complex subspace of $T_xM$. Through this identification, $\fibre/\bK$ acquires an almost complex structure and a Riemannnian metric, which are compatible (i.e. the almost complex structure is an isometry of the metric on each tangent space). The associated $2$-form is $\wred$ because $\wred$ is the projection of $\w$ and $\w=g(J\,\cdot\,,\,\cdot\,)$ on $H_x$.
\item It remains to show that the almost complex structure $J^{\mathrm{red}}$ is integrable. Since we already know that $\wred$ is symplectic and compatible with $J^{\mathrm{red}}$, it suffices to show that the Nijenhuis tensor $N_{J^{\mathrm{red}}}$ of $J^{\mathrm{red}}$ vanishes (see for instance \cite{McDuff}, Lemma 4.15). Let $X,Y$ be two projectable vector fields on $\fibre$. (i.e. $k_*X=X$ and $k_*Y=Y$ for all $k\in\bK$). Since the action is K\"ahler, it commutes to $J$, so $JX$ and $JY$ are also projectable. So is the bracket of two projectable vector fields. Moreover, by definition of $J^{\mathrm{red}}$, one has $J^{\mathrm{red}}p_* X = p_* JX$ and $J^{\mathrm{red}}p_*Y=p_*JY$. So 

\begin{eqnarray*}
N_{J^{\mathrm{red}}}(p_*X,p_*Y) & = & [J^{\mathrm{red}}p_*X,J^{\mathrm{red}}p_*Y] - J^{\mathrm{red}} [J^{\mathrm{red}}p_*X,p_*Y]  \\ && - J^{\mathrm{red}}[p_*X,J^{\mathrm{red}}p_*Y] - [p_*X,p_*Y] \\
& = & p_*([JX,JY]-J[JX,Y]-J[X,JY] - [X,Y]) \\
& = & p_* N_J(X,Y)
\end{eqnarray*}

\noindent but $N_J=0$ on $M$, so $N_{J^{\mathrm{red}}}=0$.
\end{enumerate}
\end{proof}

\noindent A wonderful account of reduction theory is given in \cite{HKLR}.

\subsection{Donaldson's Theorem}\label{section_on_Donaldson_thm}

In \cite{Don_NS}, Donaldson proposed a differential-geo\-metric proof of the celebrated Narasimhan-Seshadri theorem (\cite{NS}) which magnificiently complemented the symplectic approach to holomorphic vector bundles on a curve of Atiyah and Bott. Donaldson's theorem echoes, in an infinite-dimensional setting, a result by Kempf and Ness, relating semi-stable closed orbits of the action of a complex reductive group to the action of a maximal compact sub-group of that group. Thanks to a differential-geometric characterisation of stability, Donaldson's theorem establishes a homeomorphism between the moduli space $\cM_{\Si_g}(r,d)$ and the symplectic quotient $F^{-1}(\{\ast i2\pi\frac{d}{r\Id_E}\}) / \cG_h$.

\begin{theorem}[Donaldson, \cite{Don_NS}]
Fix a smooth Hermitian vector bundle $(E,h)$ of rank $r$ and degree $d$. Let $\cE$ be a holomorphic vector bundle of rank $r$ and degree $d$, and let $O(\cE)$ be the corresponding orbit of unitary connections on $(E,h)$. Then $\cE$ is stable if and only if $O(\cE)$ contains a unitary connection $A$ satisfying:
\begin{enumerate}
\item $\mathrm{Stab}_{\cG_E}(A)\simeq\C^*$.
\item $F_A=\ast i 2\pi \frac{d}{r} \Id_E$.
\end{enumerate}
Moreover, such a connection, if it exists, is unique up to an element of the \emph{unitary} gauge group $\cG_h$.
\end{theorem}

\noindent Indeed, since we know that isomorphism classes of holomorphic vector bundles of rank $r$ and degree $d$ are in one-to-one correspondence with \textit{complex} gauge group orbits of unitary connections on $(E,h)$, it seems natural to look for which unitary connections or more accurately, which \textit{orbits} of unitary connections, correspond to isomorphism classes of \textit{stable} holomorphic vector bundles of rank $r$ and degree $d$. Donaldson's theorem states that these orbits are precisely the complex gauge orbits of unitary connections which are both \textit{irreducible} (condition (1): $\cE=(E,A)$ is an indecomposable holomorphic vector bundle) and \textit{minimal Yang-Mills connections} (condition (2): $A$ is an absolute minimum of the Yang-Mills functional $A\lmt \int_{\Si_g} \|F_A\|^2 \vol_{\Si_g}$, see \cite{AB,Don_NS}). Moreover, any such complex gauge orbit contains a unique \textit{unitary} gauge orbit.

\begin{corollary}[The Narasimhan-Seshadri theorem, \cite{NS}]
Graded isomorphism classes of poly-stable vector bundles of rank $r$ and degree $d$ are in one-to-one correspondence with unitary gauge orbits of minimal Yang-Mills connections: $$\cM_{\Si_g}(r,d) \simeq F^{-1}\left(\{\ast i 2\pi \frac{d}{r} \Id_E\}\right) / \cG_h\, .$$
\end{corollary}

\subsection{Exercises}~

\begin{multicols}{2}

\begin{exercise}
Let $(M,J,g,\w)$ be a K\"ahler manifold and let $f:M\lra \R$ be a smooth function.  We recall that $g$ and $\w$ are non-degenerate, so define isomorphisms between $1$-forms and vector fields. The \textit{symplectic gradient} of $f$ is the vector field $\nabla^{\mathrm{symp}}_f$ defined by $$\w(\nabla^{\mathrm{symp}}_f,\,\cdot\,) = df.$$ The \textit{Riemannian gradient} of $f$ is the vector field $\nabla^{\mathrm{Riem}}_f$ defined by $$g(\nabla^{\mathrm{Riem}}_f,\,\cdot\,) = df.$$ 
\textbf{a.} Show that $f$ is constant on the integral curves of $\nabla^{\mathrm{symp}}_f$ (i.e. $df(\nabla^{\mathrm{symp}}_f)=0$). This translates to: the symplectic gradient of $f$ is tangent to the level sets of $f$ (when these are level manifolds).
\textbf{b.} Show that $$\nabla^{\mathrm{Riem}}_f = J\, \nabla^{\mathrm{symp}}_f$$ Recall that, indeed, the Riemannian gradient is orthogonal to the level manifolds of a function (it indicates the directions in which $f$ most rapidly increases).
\end{exercise}

\begin{exercise}\label{check_def_omega_red}
Show that the $2$-form $\wred$ defined in \eqref{def_omega_red} is a well-defined differential $2$-form on $\fibre/\bK$.
\end{exercise}

\begin{exercise}
Generalise Theorem \ref{thm:Kahler_red} to the quotient $$\mu^{-1}(\{\xi\})/\bK_{\xi}$$ where $\xi\in\fk^*$ no longer lies in the centre of $\fk^*$ and $\bK_{\xi}$ is the stabiliser of $\xi$ for the co-adjoint action.
\end{exercise}

\begin{exercise}
Burn the present notes and read the articles of Atiyah-Bott, Harder-Narasimhan, Donaldson and Daskalopoulos instead.
\end{exercise}

\end{multicols}


\begin{thebibliography}{HKLR87}

\bibitem[AB83]{AB}
M.~F. Atiyah and R.~Bott.
\newblock The {Y}ang-{M}ills equations over {R}iemann surfaces.
\newblock {\em Philos. Trans. Roy. Soc. London Ser. A}, 308(1505):523--615,
  1983.

\bibitem[Ati57]{Atiyah_elliptic_curves}
M.~F. Atiyah.
\newblock Vector bundles over an elliptic curve.
\newblock {\em Proc. London Math. Soc. (3)}, 7:414--452, 1957.

\bibitem[BT82]{Bott-Tu}
Raoul Bott and Loring~W. Tu.
\newblock {\em Differential forms in algebraic topology}, volume~82 of {\em
  Graduate Texts in Mathematics}.
\newblock Springer-Verlag, New York, 1982.

\bibitem[Das92]{Dask}
Georgios~D. Daskalopoulos.
\newblock The topology of the space of stable bundles on a compact {R}iemann
  surface.
\newblock {\em J. Differential Geom.}, 36(3):699--746, 1992.

\bibitem[DK90]{DK}
S.~K. Donaldson and P.~B. Kronheimer.
\newblock {\em The geometry of four-manifolds}.
\newblock Oxford Mathematical Monographs. The Clarendon Press Oxford University
  Press, New York, 1990.
\newblock Oxford Science Publications.

\bibitem[Don83]{Don_NS}
S.~K. Donaldson.
\newblock A new proof of a theorem of {N}arasimhan and {S}eshadri.
\newblock {\em J. Differential Geom.}, 18(2):269--277, 1983.

\bibitem[For91]{Forster}
Otto Forster.
\newblock {\em Lectures on {R}iemann surfaces}, volume~81 of {\em Graduate
  Texts in Mathematics}.
\newblock Springer-Verlag, New York, 1991.
\newblock Translated from the 1977 German original by Bruce Gilligan, Reprint
  of the 1981 English translation.

\bibitem[Fri98]{Friedman}
Robert Friedman.
\newblock {\em Algebraic surfaces and holomorphic vector bundles}.
\newblock Universitext. Springer-Verlag, New York, 1998.

\bibitem[G\'01]{Gomez}
Tom\'as~L. G\'omez.
\newblock {Algebraic stacks.}
\newblock {\em Proc. Indian Acad. Sci., Math. Sci.}, 111(1):1--31, 2001.

\bibitem[GH94]{Griffiths-Harris}
Phillip Griffiths and Joseph Harris.
\newblock {\em {Principles of algebraic geometry. 2nd ed.}}
\newblock {New York, NY: John Wiley \& Sons Ltd.}, 1994.

\bibitem[Gro57]{Grot_P1}
A.~Grothendieck.
\newblock {Sur la classification des fibres holomorphes sur la sph\`ere de
  Riemann.}
\newblock {\em Am. J. Math.}, 79:121--138, 1957.

\bibitem[Gun66]{Gunning_RS}
R.C. Gunning.
\newblock {\em {Lectures on Riemann surfaces.}}
\newblock {Princeton, NJ: Princeton University Press}, 1966.

\bibitem[Hat02]{Hatcher}
Allen Hatcher.
\newblock {\em Algebraic topology}.
\newblock Cambridge University Press, Cambridge, 2002.

\bibitem[HKLR87]{HKLR}
N.~J. Hitchin, A.~Karlhede, U.~Lindstr{\"o}m, and M.~Ro{\v{c}}ek.
\newblock Hyper-{K}\"ahler metrics and supersymmetry.
\newblock {\em Comm. Math. Phys.}, 108(4):535--589, 1987.

\bibitem[HN75]{HN}
G.~Harder and M.~S. Narasimhan.
\newblock On the cohomology groups of moduli spaces of vector bundles on
  curves.
\newblock {\em Math. Ann.}, 212:215--248, 1974/75.

\bibitem[Hus93]{Husemoller}
Dale~H. Husemoller.
\newblock {\em {Fibre bundles. 3rd ed.}}
\newblock {Berlin: Springer-Verlag}, 1993.

\bibitem[KN96]{Kobayashi_Nomizu}
Shoshichi Kobayashi and Katsumi Nomizu.
\newblock {\em Foundations of differential geometry. {V}ol. {I}}.
\newblock Wiley Classics Library. John Wiley \& Sons Inc., New York, 1996.
\newblock Reprint of the 1963 original, A Wiley-Interscience Publication.

\bibitem[Kob87]{Kobayashi}
Shoshichi Kobayashi.
\newblock {\em {Differential geometry of complex vector bundles.}}
\newblock {Princeton, NJ: Princeton University Press; Tokyo: Iwanami Shoten
  Publishers}, 1987.

\bibitem[LPV85]{VLP}
J.~Le~Potier and J.L. Verdier.
\newblock Vari\'et\'e de modules de fibr\'es stables sur une surface de
  {R}iemann: r\'esultats d'{A}tiyah et {B}ott.
\newblock In {\em Moduli of stable bundles over algebraic curves ({P}aris,
  1983)}, volume~54 of {\em Progr. Math.}, pages 5--28. Birkh\"auser Boston,
  Boston, MA, 1985.

\bibitem[MFK93]{Mumford_GIT}
D.~Mumford, J.~Fogarty, and F.~Kirwan.
\newblock {\em {Geometric invariant theory. 3rd enl. ed.}}
\newblock {Berlin: Springer-Verlag}, 1993.

\bibitem[MS98]{McDuff}
Dusa McDuff and Dietmar Salamon.
\newblock {\em Introduction to symplectic topology}.
\newblock Oxford Mathematical Monographs. The Clarendon Press Oxford University
  Press, New York, second edition, 1998.

\bibitem[Muk03]{Mukai}
Shigeru Mukai.
\newblock {\em {An introduction to invariants and moduli. Transl. by W. M.
  Oxbury.}}
\newblock {Cambridge: Cambridge University Press}, 2003.

\bibitem[Mum63]{Mumford_VB}
David Mumford.
\newblock Projective invariants of projective structures and applications.
\newblock In {\em Proc. {I}nternat. {C}ongr. {M}athematicians ({S}tockholm,
  1962)}, pages 526--530. Inst. Mittag-Leffler, Djursholm, 1963.

\bibitem[New09]{Newstead_GIT}
P.~E. Newstead.
\newblock Geometric invariant theory.
\newblock In {\em Moduli spaces and vector bundles}, volume 359 of {\em London
  Math. Soc. Lecture Note Ser.}, pages 99--127. Cambridge Univ. Press,
  Cambridge, 2009.

\bibitem[NS65]{NS}
M.~S. Narasimhan and C.~S. Seshadri.
\newblock Stable and unitary vector bundles on a compact {R}iemann surface.
\newblock {\em Ann. of Math. (2)}, 82:540--567, 1965.

\bibitem[Ses67]{Seshadri}
C.~S. Seshadri.
\newblock Space of unitary vector bundles on a compact {R}iemann surface.
\newblock {\em Ann. of Math. (2)}, 85:303--336, 1967.

\bibitem[Ses82]{Seshadri_ast}
C.~S. Seshadri.
\newblock {\em Fibr\'es vectoriels sur les courbes alg\'ebriques}, volume~96 of
  {\em Ast\'erisque}.
\newblock Soci\'et\'e Math\'ematique de France, Paris, 1982.
\newblock Notes written by J.-M. Drezet from a course at the {\'E}cole Normale
  Sup{\'e}rieure, June 1980.

\bibitem[Sha77]{Shatz}
Stephen~S. Shatz.
\newblock The decomposition and specialization of algebraic families of vector
  bundles.
\newblock {\em Compositio Math.}, 35(2):163--187, 1977.

\bibitem[Ste51]{Steenrod}
Norman Steenrod.
\newblock {\em {The topology of fibre bundles.}}
\newblock {Princeton, NJ: Princeton University Press}, 1951.

\bibitem[Tha97]{Thaddeus}
Michael Thaddeus.
\newblock {An introduction to the topology of the moduli space of stable
  bundles on a Riemann surface.}
\newblock {Andersen, J{\o}rgen Ellegaard (ed.) et al., Geometry and physics.
  Proceedings of the conference at Aarhus University, Aarhus, Denmark, 1995.
  New York, NY: Marcel Dekker. Lect. Notes Pure Appl. Math. 184, 71-99
  (1997).}, 1997.

\bibitem[Tho06]{Thomas_GIT}
R.~P. Thomas.
\newblock Notes on {GIT} and symplectic reduction for bundles and varieties.
\newblock In {\em Surveys in differential geometry. {V}ol. {X}}, volume~10 of
  {\em Surv. Differ. Geom.}, pages 221--273. Int. Press, Somerville, MA, 2006.

\bibitem[Wel08]{Wells}
Raymond O.~Jun. Wells.
\newblock {\em {Differential analysis on complex manifolds. With a new appendix
  by Oscar Garcia-Prada. 3rd ed.}}
\newblock {New York, NY: Springer}, 2008.

\end{thebibliography}
\end{document}